\documentclass[11pt,a4paper]{amsart}
\usepackage{amsmath,amssymb,color,multicol,setspace}
\usepackage[utf8,utf8x]{inputenc}

\usepackage[pdftitle = {Frieze\ patterns\ over\ integers\ and\ other\ subsets\ of\ the\ complex\ numbers}]{hyperref}
\usepackage{longtable}
\usepackage{xy,amscd}
\usepackage{epsfig}
\usepackage{rotating}
\usepackage{xspace}
\usepackage{enumitem}
\xyoption{all}
\usepackage{tikz}
\usetikzlibrary{arrows,decorations.pathmorphing,decorations.pathreplacing,positioning,shapes.geometric,shapes.misc,decorations.markings,decorations.fractals,calc,patterns}
\usepackage{xcolor}
\usepackage{comment}

\newcommand{\blue}{\color{blue}}
\newcommand{\red}{\color{red}}

\usepackage{fullpage}

\newtheorem{Lemma}{Lemma}[section]
\newtheorem{Theorem}[Lemma]{Theorem}
\newtheorem{Proposition}[Lemma]{Proposition}
\newtheorem{Corollary}[Lemma]{Corollary}

\theoremstyle{definition}
\newtheorem{Definition}[Lemma]{Definition}

\newtheorem{Remark}[Lemma]{Remark}

\newtheorem{Example}[Lemma]{Example}

\numberwithin{equation}{section}

\newcommand{\Q}{\mathbb{Q}}
\newcommand{\Z}{\mathbb{Z}}

\newcommand{\GE}{\operatorname{GE}}

\newcommand{\sps}{\makebox[20pt]{\ }}

\title{Non-commutative frieze patterns over quaternion algebras and other normed division rings}

\author{Michael~Cuntz}
\address{Michael Cuntz, Leibniz Universit\"at Hannover,
Institut f\"ur Algebra, Zahlentheorie und Dis\-krete Mathematik,
Fakult\"at f\"ur Mathematik und Physik,
Welfengarten 1,
D-30167 Hannover, Germany}
\email{cuntz@math.uni-hannover.de}
\urladdr{https://www.iazd.uni-hannover.de/de/cuntz}

\author{Thorsten~Holm}
\address{Thorsten Holm, Leibniz Universit\"at Hannover,
Institut f\"ur Algebra, Zahlentheorie und Dis\-krete Mathematik,
Fakult\"at f\"ur Mathematik und Physik,
Welfengarten 1,
D-30167 Hannover, Germany}
\email{holm@math.uni-hannover.de}
\urladdr{https://www.iazd.uni-hannover.de/de/holm}

\author{Peter J{\o}rgensen}
\address{Peter J{\o}rgensen, 
Department of Mathematics,
Aarhus University,
Ny Munkegade 118,
8000 Aarhus C,
Denmark}
\email{peter.jorgensen@math.au.dk}
\urladdr{https://sites.google.com/view/peterjorgensen}



\begin{document}

\begin{abstract}
Non-commutative friezes have been introduced by Berenstein and Retakh and studied further
by the authors. In this paper we consider non-commutative friezes over normed division rings, 
like for instance Hamilton's quaternions or more general quaternion algebras. We address the fundamental
question in the theory of friezes of whether over a certain subset there are finitely or infinitely many 
non-commutative friezes (with 1's on the boundary) for any height. As an application of a theorem bounding the norm
of quiddity entries we deduce that for every norm-finite subset of a normed division ring there are only
finitely many such non-commutative friezes for every height. In particular this result applies 
to the Lipschitz quaternions and the Hurwitz quaternions of Hamilton's quaternions. We then study more
generally non-commutative friezes over Lipschitz subrings of non-split quaternion algebras 
$(a,b)_{\mathbb{Q}}$. We determine the
frieze subrings for all $a,b<0$, and as a consequence we
see that all such non-commutative friezes are known if $a\le -4$ and $b\le -4$.
\end{abstract}

\maketitle

\section{Introduction}
Frieze patterns of numbers have been introduced by Coxeter \cite{Cox71} in 1971. Soon afterwards, Conway and
Coxeter \cite{CC73} have developed the theory and, among other things, have shown that frieze patterns of positive
integers are in bijection with triangulations of polygons. This connects frieze patterns with a plethora 
of other objects in mathematics which are counted by Catalan numbers. The interest in frieze patterns
grew further around 2000 when Fomin and Zelevinsky introduced cluster algebras. It turned out that the 
exchange relations in cluster algebras of Dynkin type A correspond to the defining relations in Coxeter's frieze patterns so that
the generators of cluster algebras, the cluster variables, appear as entries in certain frieze patterns. Since then 
an extensive literature on frieze patterns has evolved and frieze patterns nowadays form a nexus between different 
mathematical areas like algebra, combinatorics, geometry and number theory.

Classically, the entries 
in frieze patterns come from some commutative ring, this also reflects the fact that cluster algebras 
are commutative rings. Only recently, the concept of frieze patterns has been 
taken to a non-commutative setting. This has been initiated by Berenstein and Retakh in their 
remarkable paper on non-commutative marked surfaces \cite{BR18}. Building on their notion of non-commutative
polygons, we developed a theory of non-commutative frieze patterns; we show in \cite{CHJ24}, \cite{CHJ25}
how several classic results for (commutative) frieze patterns can be transferred to non-commutative 
versions, for instance the very definition of frieze patterns by a small set of local relations, a formula for
frieze determinants, results on gluing friezes and, most importantly, a $T$-path formula expressing 
a non-commutative version of the Laurent phenomenon. Such results on non-commutative frieze patterns 
could be seen as a hint that there might even be some non-commutative versions of cluster algebras. So far,
it seems that this has not been achieved in general but specific non-commutative cluster algebras have recently been 
introduced and studied, see \cite{BR18}, \cite{GK21}, \cite{GKW24}, \cite{GKNW24}.

A non-commutative frieze on a polygon $\mathcal{P}$ is a map from the diagonals, that is ordered
pairs of vertices, to the invertible elements of a ring $D$ such that two types of relations are satisfied,
the triangle relations and the exchange relations (see Definition \ref{def:ncpolygon} for details).  
A crucial difference to a classic frieze is that for a non-commutative frieze the diagonals 
exist in both directions, so in the corresponding non-commutative frieze values $c_{i,j}$ and $c_{j,i}$
can be different. 
Every classic frieze over a commutative ring is also a non-commutative frieze; the triangle relations 
are trivial in the commutative case and the exchange relations become the well-known Ptolemy relations. 
However, for a non-commutative frieze the triangle relations and the exchange relations involve
inverses of elements from $D$ and they do not cancel out (as in the commutative case). Therefore,
it is most natural to consider non-commutative friezes over division rings and this is the topic of this paper. 

Starting from the ring of Hamilton quaternions we consider more generally non-commutative friezes 
over quaternion algebras. Quaternion algebras are 4-dimensional central simple algebras over a field,
they are applied in number theory in the context of quadratic forms and Brauer groups.
Quaternion algebras $(a,b)_F$ depend on two parameters $a,b$ from the field $F$ and they
can be defined explicitly by relations similar to Hamilton quaternions.
They provide a
rich source of division rings, actually each quaternion algebra is either a division rings or isomorphic
to the ring of $2\times 2$-matrices over the underlying field (in characteristic not 2). In Section \ref{sec:normed} we recall 
some basic definitions and properties of quaternion algebras which are needed later. 
For more details we refer to the extensive textbook by Voight \cite{Voight}. 

A crucial feature of quaternion algebras is that they admit a norm form $N:(a,b)_F\to F$ which 
can be used to define a norm on certain quaternion algebras. In this paper we more generally
consider non-commutative friezes over normed division rings (that is, division rings $D$ with a norm
function $\lVert .\rVert:D\to \mathbb{R}_{\ge 0}$). This includes non-commutative friezes over Hamilton
quaternions as a special case, but opens a much broader perspective. 

An important and fundamental question in the theory of friezes is whether for a given subset of a ring
there are finitely many or infinitely many friezes for each
height. We address this question for friezes over
normed division rings
where all boundary entries are equal to 1 (which is by far the most intensively studied type of friezes
in the literature). 
However, our first main result deals more
generally with non-commutative friezes with 
arbitrary boundary entries.
\medskip

\noindent
{\bf Theorem \ref{lem:finite}}. 
{\em Let $D$ be a normed division ring. 
Let $R\subseteq D \setminus\{0\}$ be a subset such that 
$$M:=\inf\{\lVert x\rVert\,:\,x\in R\}>0.
$$
Let $\mathcal{C}=(c_{i,j})$ be a non-commutative frieze 
on an $(n+3)$-gon with 
$n\ge 1$
over $R$ and set
$$P= \max \{ \lVert c_{i,i+1}\rVert, \lVert c_{i+1,i}\rVert \,:\, 0\le i\le n+2\},
$$
the maximal value of the norm of a boundary entry. 
Then the norm of every quiddity entry $c_{j,j+2}$ and $c_{j+2,j}$ of $\mathcal{C}$ is 
at most 
$$\frac{P^2(M+nP)}{M^2}.
$$
}

\bigskip

As a consequence we obtain that for many subsets of normed division rings there are only finitely many 
non-commutative friezes for any height. For the definition of {\em norm-finite subset} see
Definition \ref{def:normedsubset}. In particular the following result applies to prominent subsets 
of Hamilton's quaternions like the Lipschitz quaternions and the Hurwitz quaternions. 
\smallskip

\noindent
{\bf Corollary \ref{cor:finitelymany}.}
{\em Let $D$ be a normed division ring and $R\subseteq D\setminus \{0\}$ a norm-finite subset. 
Then for each $n\in \mathbb{N}$ there are only finitely many non-commutative friezes over $R$ of height $n$
with all boundary entries 1. 
}
\medskip

In the final section we consider specifically non-split quaternion algebras over the rational numbers, that is,
quaternion algebras of the form $(a,b)_{\mathbb{Q}}$ with negative rational numbers $a,b$. We define 
(see Definition \ref{def:lipschitz}) the {\em Lipschitz subring} as
$$(a,b)_{\mathbb{Z}} = \{\alpha + \beta\mathrm{i}+\gamma\mathrm{j}+\delta\mathrm{k}\,|\,
\alpha,\beta,\gamma,\delta\in \mathbb{Z}\}.
$$
We observe that the Lipschitz subring is a norm-finite subset of the normed division ring $(a,b)_{\mathbb{Q}}$. 
So we can apply Corollary \ref{cor:finitelymany} and deduce that there are only finitely many non-commutative
friezes with 1s on the boundary over the Lipschitz subring. 

A very natural question then is whether one can find or classify all these non-commutative friezes. Roughly speaking,
this is possible in very many cases, namely unless one of $a$ and $b$ takes values in $\{-3,-2,-1\}$. 
\smallskip

\noindent
{\bf Definition \ref{def:friezesubring}.}
Let $R$ be a ring. The {\em non-commutative frieze subring} of $R$ is defined as the ring $R^{\circ}$ 
generated by all entries of all non-commutative frieze patterns over $R$ with ones on the boundary. 
\medskip

With this notion we can then state another main result. 
\smallskip

\noindent
{\bf Theorem \ref{thm:abcirc}.}
{\em Let $a<0$ and $b<0$ be integers. For the non-commutative frieze subring of the Lipschitz
quaternion ring we get 
$$(a,b)_{\mathbb{Z}}^{\circ} = \left\{
\begin{array}{ll} \mathbb{Z} & \mbox{ if $a\le -4$ and $b\le -4$} \\
(a,b)_{\mathbb{Z}} & \mbox{ if $a\ge -3$ and $b\ge -3$}
\end{array} \right.
$$
}

Note that this result says that if $a\le -4$ and $b\le -4$ the only non-commutative friezes over
the Lipschitz subring $(a,b)_{\mathbb{Z}}$ are the classic friezes over the integers $\mathbb{Z}$,
that is, the Conway-Coxeter friezes and the related twisted versions obtained by multiplying 
every second diagonal by $-1$. It seems to be a hard question to find or list all non-commutative 
friezes if $a$ or $b$ is in $\{-3,-2,-1\}$.
Based on non-AI computer 
experiments we conjectured that 
$(a,b)_{\mathbb{Z}}^{\circ}$ is equal to 
$\mathbb{Z}[\mathrm{i}]$ (if $a\in \{-3,-2,-1\}$ and 
$b\le -4$)
or $\mathbb{Z}[\mathrm{j}]$ (if $a\le -4$ and $b\in \{-3,-2,-1\}$).
The Appendix contains a proof of this
conjecture found by ChatGPT 5.6 Pro. Combining this with 
Theorem \ref{thm:abcirc} we get the following complete
answer for the non-commutative frieze subrings
of the Lipschitz quaternions.
\smallskip

\noindent
{\bf Theorem \ref{cor:classification}.}
{\em Let $a<0$ and $b<0$ be integers. For the non-commutative frieze subring of the Lipschitz
quaternion ring we have 
$$(a,b)_{\mathbb{Z}}^{\circ} = \left\{
\begin{array}{ll} \mathbb{Z} & \mbox{ if $a\le -4$ and $b\le -4$}, \\
(a,b)_{\mathbb{Z}} & \mbox{ if $a\ge -3$ and $b\ge -3$}, \\
\mathbb{Z}[\mathrm{i}] & \mbox{ if $a\in \{-1,-2,-3\}$
and $b\le -4$}, \\
\mathbb{Z}[\mathrm{j}] & \mbox{ if $a\le -4$ and $b\in \{-1,-2,-3\}$}.
\end{array} \right.
$$
}

\section{Normed division rings and quaternion algebras} \label{sec:normed}

\begin{Definition} \label{def:normeddivring}
A {\em normed division ring} $D$ is an associative ring with unity such that 
each non-zero element is invertible in $D$ and there exists a norm function
$\lVert .\rVert:D\to \mathbb{R}_{\ge 0}$, satisfying the following axioms.
\begin{enumerate}
\item[{(N1)}] $\lVert a\rVert =0$ $\Longleftrightarrow$ $a=0$.
\item[{(N2)}] $\lVert a\cdot b\rVert = \lVert a\rVert
\cdot \lVert b\rVert$ for all $a,b\in D$.
\item[{(N3)}] $\lVert a+b\rVert \le \rVert a\rVert +\lVert b\rVert$ for all $a,b\in D$.
\end{enumerate}
\end{Definition}

\medskip

\begin{Example} \label{ex:hamilton}
A famous example of a normed division ring is the Hamilton quaternions 
$$\mathbb{H} = \{\alpha+\beta\mathrm{i}+\gamma\mathrm{j}+\delta\mathrm{k}\,|\,
\alpha,\beta,\gamma,\delta\in \mathbb{R}\}.
$$
This is a 4-dimensional real vector space and the product between basis
elements is given in the following table (and extended to arbitrary elements of $\mathbb{H}$ 
by associativity and distributivity): 
$$\begin{array}{c|c|c|c}
  & \mathrm{i} & \mathrm{j} & \mathrm{k} \\
 \hline
 \mathrm{i} & -1 & \mathrm{k} & -\mathrm{j} \\
 \hline
 \mathrm{j} & -\mathrm{k} & -1 & \mathrm{i} \\
 \hline
 \mathrm{k} & \mathrm{j} & -\mathrm{i} & -1
 \end{array}
 $$
For an element $q=\alpha+\beta\mathrm{i}+\gamma\mathrm{j}+\delta\mathrm{k}\in \mathbb{H}$
its {\em conjugate} is the element
$$\overline{q} =  \alpha-\beta\mathrm{i}-\gamma\mathrm{j}-\delta\mathrm{k}\in \mathbb{H}.
$$

A norm function on $\mathbb{H}$ is given by
$$\lVert \alpha+\beta\mathrm{i}+\gamma\mathrm{j}+\delta\mathrm{k}\rVert = 
\sqrt{\alpha^2+\beta^2+\gamma^2+\delta^2} = \sqrt{q\,\overline{q}}.
$$
(This is indeed a norm function in the sense of Definition \ref{def:normeddivring}.
(N1) is immediate from the definition. For (N2), let $q_1,q_2\in \mathbb{H}$. In 
Proposition \ref{prop:conjnormprop} below we will show in a more general setting that
$\overline{q_1q_2}=\overline{q_2}\,\overline{q_1}$. Using this we get 
$$\lVert q_1\,q_2\rVert^2 = q_1q_2\overline{q_1q_2} = q_1q_2\overline{q_2}\,\overline{q_1}
= q_1 \lVert q_2\rVert^2\,\overline{q_1} = q_1 \,\overline{q_1}\,\lVert q_2\rVert^2 = \lVert q_1\rVert^2\cdot \lVert q_2\rVert^2
$$
and (N2) follows by taking square roots. For (N3), note that the norm on $\mathbb{H}$ 
is the usual Euclidean norm when $\mathbb{H}$ is viewed as a 4-dimensional 
real vector space, in particular, the triangle inequality (N3) holds.)

The inverse of a non-zero quaternion is then given by 
$$(\alpha+\beta\mathrm{i}+\gamma\mathrm{j}+\delta\mathrm{k})^{-1} = 
\frac{1}{\lVert \alpha+\beta\mathrm{i}+\gamma\mathrm{j}+\delta\mathrm{k}\rVert^2}
(\alpha-\beta\mathrm{i}-\gamma\mathrm{j}-\delta\mathrm{k}) \in \mathbb{H}.
$$
Hence, $\mathbb{H}$ is a normed division ring.
\medskip

There are many subrings of $\mathbb{H}$ which are also normed division rings (with the norm induced
from $\mathbb{H}$). For any subfield $K\subseteq \mathbb{R}$, the subset
$$\mathbb{H}_K:=\{\alpha+\beta\mathrm{i}+\gamma\mathrm{j}+\delta\mathrm{k}\,|\,
\alpha,\beta,\gamma,\delta\in K\} \subseteq \mathbb{H}
$$ 
is a normed division ring. In fact, $\mathbb{H}_K$ is clearly a subring of $\mathbb{H}$, and for the 
inverses we have
$$(\alpha+\beta\mathrm{i}+\gamma\mathrm{j}+\delta\mathrm{k})^{-1}
= \frac{1}{\alpha^2+\beta^2+\gamma^2+\delta^2} (\alpha-\beta\mathrm{i}-\gamma\mathrm{j}-
\delta\mathrm{k})\in \mathbb{H}_K
$$
since $K$ is a subfield of $\mathbb{R}$.
\end{Example}

\medskip

A natural and important generalization of Hamilton's quaternions is {\em quaternion algebras}
which we now define and consider. Among these we will also find further examples 
of normed division rings. Quaternion algebras can be defined over arbitrary base fields. They
have applications in number theory, in the context of quadratic forms and Brauer groups. 
We present here only some fundamental definitions and properties which we shall later 
need for considering non-commutative friezes over quaternion algebras. 
For more details on quaternion algebras we refer to the textbook by Voight
\cite{Voight} or the notes by Conrad \cite{Conrad}. 
\bigskip

Let $F$ be a field of characteristic not 2. 

\begin{Definition} \label{def:quatalg}
A {\em quaternion algebra} over $F$ is a 4-dimensional algebra over $F$ with a vector space basis
$\{1,\mathrm{i},\mathrm{j},\mathrm{k}\}$ such that the following conditions are satisfied:
\begin{enumerate}
\item[{(i)}] $\mathrm{i}^2=a$ and $\mathrm{j}^2=b$ for 
certain fixed $a,b\in F\setminus\{0\}$,
\item[{(ii)}] $\mathrm{i}\mathrm{j}=\mathrm{k}=-\mathrm{j}\mathrm{i}$,
\item[{(iii)}] all elements of $F$ commute with $\mathrm{i}$ and $\mathrm{j}$.  
\end{enumerate}
Notation for such a quaternion algebra: $(a,b)_F$. 
\end{Definition}

\begin{Remark} \label{rem:quatprod}
From the relations in Definition \ref{def:quatalg} one can deduce all products between the elements 
of the vector space basis, as given in the following table.
$$\begin{array}{c|c|c|c}
  & \mathrm{i} & \mathrm{j} & \mathrm{k} \\
 \hline
 \mathrm{i} & a & \mathrm{k} & a\mathrm{j} \\
 \hline
 \mathrm{j} & -\mathrm{k} & b & -b\mathrm{i} \\
 \hline
 \mathrm{k} & -a\mathrm{j} & b\mathrm{i} & -ab
 \end{array}
 $$
 As a consequence one has the following formula for the product of two arbitrary elements
 $q_1=\alpha_1+\beta_1 \mathrm{i} + \gamma_1 \mathrm{j} + \delta_1 \mathrm{k}$ and
  $q_2=\alpha_2+\beta_2 \mathrm{i} + \gamma_2 \mathrm{j} + \delta_2 \mathrm{k}$ in $(a,b)_F$:
\begin{eqnarray*}
q_1q_2 & = & (\alpha_1\alpha_2 + a\beta_1\beta_2 + b \gamma_1\gamma_2 -ab\delta_1\delta_2) \\
& & + (\alpha_1\beta_2 + \beta_1\alpha_2 - b \gamma_1\delta_2 + b\delta_1\gamma_2) \mathrm{i} \\
& & + ( \alpha_1\gamma_2 + a\beta_1\delta_2 + \gamma_1\alpha_2 - a\delta_1\beta_2) \mathrm{j} \\
& & + (\alpha_1\delta_2 +\beta_1\gamma_2 - \gamma_1\beta_2 +\delta_1\alpha_2) \mathrm{k}.
\end{eqnarray*}
\end{Remark}

\begin{Example}
Hamilton's quaternions $\mathbb{H}$ (as introduced in Example \ref{ex:hamilton})
appear for the field $F=\mathbb{R}$ and $a=b=-1$,
i.e. $\mathbb{H}=(-1,-1)_{\mathbb{R}}$. 
\end{Example}

\begin{Definition} \label{def:norm}
Let $(a,b)_F$ be a quaternion algebra. 
For every element $q=\alpha + \beta \mathrm{i}+\gamma \mathrm{j} +\delta \mathrm{k}\in (a,b)_F$ its {\em conjugate} is
defined as 
$$\overline{q} = \alpha - \beta \mathrm{i} - \gamma \mathrm{j} - \delta \mathrm{k}\in (a,b)_F.
$$
The {\em norm form} on $(a,b)_{F}$ is the function 
$$N:(a,b)_F\to F,~~N(q) = \alpha^2 -a \beta^2 - b\gamma^2 + ab \delta^2.
$$ 
\end{Definition}
\smallskip

We collect some fundamental properties of the conjugates and of the norm form. 

\begin{Proposition} \label{prop:conjnormprop}
Let $(a,b)_F$ be a quaternion algebra. 
\begin{enumerate}
\item[{(i)}] For every $q,q_1,q_2\in (a,b)_F$ we have
$$\overline{q_1+q_2} = \overline{q_1} + \overline{q_2},~~\overline{q_1q_2} = \overline{q_2}\,\overline{q_1},
\mbox{~~and~~}\overline{\overline{q}}=q.
$$
Note the different ordering in the formula for the product.
\item[{(ii)}] For each $q\in (a,b)_F$ we have $N(q) = q\,\overline{q}=\overline{q}\,q$. 
\item[{(iii)}] For every $q_1,q_2\in (a,b)_F$ we have
$N(q_1q_2)=N(q_1)N(q_2)$, i.e., the norm form is multiplicative. 
\end{enumerate}
\end{Proposition}

\begin{proof} We write $q=\alpha + \beta \mathrm{i}+\gamma \mathrm{j} +\delta \mathrm{k}$ and
$q_{\ell}=\alpha_{\ell} + \beta_{\ell} \mathrm{i}+\gamma_{\ell} \mathrm{j} +\delta_{\ell} \mathrm{k}$ for $\ell=1,2$. 
\begin{enumerate}
\item[{(i)}] Direct computations yield
\begin{eqnarray*}
\overline{q_1+q_2} & = & (\alpha_{1}+\alpha_2) - (\beta_{1}+\beta_2) \mathrm{i} -
(\gamma_{1}+\gamma_2) \mathrm{j} - (\delta_{1}+\delta_2) \mathrm{k} \\
& = & (\alpha_{1} - \beta_{1} \mathrm{i}-\gamma_{1}\mathrm{j} - \delta_{1} \mathrm{k}) +
 (\alpha_{2} - \beta_{2} \mathrm{i} - \gamma_{2}\mathrm{j} - \delta_{2} \mathrm{k}) \,=\, 
 \overline{q_1}+\overline{q_2}
\end{eqnarray*}
and using the product formula in Remark \ref{rem:quatprod} we also get 
\begin{eqnarray*}
\overline{q_2}\,\overline{q_1} & = & (\alpha_2\alpha_1 + a\beta_2\beta_1 + b \gamma_2\gamma_1 -ab\delta_2\delta_1) \\
& & + (-\alpha_2\beta_1 - \beta_2\alpha_1 - b \gamma_2\delta_1 + b\delta_2\gamma_1) \mathrm{i} \\
& & + ( -\alpha_2\gamma_1 + a\beta_2\delta_1 - \gamma_2\alpha_1 - a\delta_2\beta_1) \mathrm{j} \\
& & + (-\alpha_2\delta_1 +\beta_2\gamma_1 - \gamma_2\beta_1 +\delta_2\alpha_1) \mathrm{k} \\
& = & \overline{q_1q_2}.
\end{eqnarray*}
The formula $\overline{\overline{q}}=q$ is immediate from the definition of the conjugation.
\item[{(ii)}] Again using the product formula from Remark \ref{rem:quatprod} and
Definition \ref{def:norm} we obtain
\begin{eqnarray*}
q\,\overline{q} & = & (\alpha^2-a\beta^2-b\gamma^2+ab\delta^2) \\
& & + (-\alpha\beta + \beta\alpha +b\gamma\delta-b\delta\gamma) \mathrm{i} \\
& & + ( -\alpha\gamma-a\beta\delta+\gamma\alpha+a\delta\beta) \mathrm{j} \\
& & + (-\alpha\delta-\beta\gamma+\gamma\beta+\delta\alpha) \mathrm{k} \\
& = & \alpha^2-a\beta^2-b\gamma^2+ab\delta^2 \,=\, N(q). 
\end{eqnarray*}
A very similar computation shows that also $\overline{q}\,q = N(q)$.
\item[{(iii)}] By (i) and (ii) and since values of the norm form are in the field $F$ (whose elements are in the 
center of $(a,b)_F$ by Definition \ref{def:quatalg}) we get
$$N(q_1q_2) = q_1q_2\,\overline{q_1q_2} = q_1q_2 \overline{q_2}\,\overline{q_1}
= q_1 N(q_2)\,\overline{q_1} = q_1\,\overline{q_1} N(q_2) =N(q_1)N(q_2). 
$$
\end{enumerate}
\end{proof}

Whether an element in a quaternion algebra $(a,b)_F$ is invertible (with respect to multiplication)
is determined by the norm form. 

\begin{Proposition} \label{prop:invertible}
Let $(a,b)_F$ be a quaternion algebra and $N:(a,b)_F\to F$ its norm form (as in Definition \ref{def:norm}). 
An element $q\in (a,b)_F$ is invertible if and only if
$N(q)\neq 0$.
\end{Proposition}

\begin{proof}
Let $q$ be invertible, i.e., there exists $q'\in (a,b)_F$ such that $qq'=1$ and $q'q=1$. By Proposition 
\ref{prop:conjnormprop}\,(iii) it follows that
$N(q)N(q') = N(qq')=N(1)=1$. In particular, $N(q)\neq 0$. 

Conversely, suppose that $N(q)\neq 0$. Then Proposition \ref{prop:conjnormprop}\,(ii) yields
$$q\left( N(q)^{-1} \overline{q}\right) = N(q)^{-1} q\,\overline{q} = 1 = \left( N(q)^{-1} \overline{q}\right)q.
$$
Thus $q$ is invertible, with inverse  $N(q)^{-1} \overline{q}\in (a,b)_F$. 
\end{proof}

\begin{Remark} \label{rem:a1}
It is well-known that the Hamilton quaternions $\mathbb{H}=(-1,-1)_{\mathbb{R}}$ form a division 
ring. This easily follows from Proposition \ref{prop:invertible} as in this case the norm  
$$N(q)=N(\alpha + \beta \mathrm{i}+\gamma \mathrm{j} + \delta \mathrm{k}) = 
\alpha^2+\beta^2+\gamma^2+\delta^2$$
is a positive real number for all $q\neq 0$. 
\medskip

For other values of $a,b$ and other fields $F$ it is a priori unclear whether non-zero elements having norm 0
exist.  

As an example, the quaternion algebra $(a,1)_F$ is not a division ring for every field $F$ and all $a\in F$. 
In fact, in $(a,1)_F$ we have by Definition \ref{def:norm} that 
$$N(\alpha+\beta \mathrm{i} + \gamma \mathrm{j}+\delta \mathrm{k}) = \alpha^2 -a\beta ^2-\gamma^2+a\delta^2.
$$
For instance, every element with $\alpha=\gamma$ and $\beta=\delta$ has norm 0, hence is not invertible
in $(a,1)_F$. 
\end{Remark}

\medskip

For the structure of quaternion algebras there are only two fundamental possibilities, 
they are either division rings or matrix rings. A proof of the following result can be found
in the notes by Conrad \cite[Theorem 4.20]{Conrad}

\begin{Theorem} \label{thm:notdivision}
Let $F$ be a field of characteristic not 2. 
Every quaternion algebra $(a,b)_F$ which is not a division ring is isomorphic to
the ring $M(2,F)$ of $2\times 2$-matrices over $F$. 
\end{Theorem}

We recall a standard notation from the theory of quaternion algebras.

\begin{Definition} \label{def:split}
Let $F$ be a field and $a,b\in F$. The quaternion algebra $(a,b)_F$ is called {\em split} if
$(a,b)_F\cong M(2,F)$. Otherwise, $(a,b)_F$ is called a {\em non-split} quaternion algebra. 
\end{Definition}

Note that by Theorem \ref{thm:notdivision} all non-split quaternion algebras are division rings. 
\bigskip

\begin{Remark}
Let us look at quaternion algebras over some familiar fields. 
\smallskip

Over finite prime fields $\mathbb{F}_p$ for $p$ an odd prime, all quaternion algebras are split (see
\cite[Corollary 4.24]{Conrad}). 
\smallskip

Over the real numbers there are only two quaternion algebras, up to isomorphism. This follows from the
famous theorem of Frobenius \cite{Frobenius} (see also \cite[Corollary 3.5.8]{Voight} for a modern 
textbook version), 
stating that the only finite-dimensional associative real division algebras 
are $\mathbb{R}$, $\mathbb{C}$ and $\mathbb{H}$. More precisely, 
for $a,b\in \mathbb{R}\setminus \{0\}$ we have that 
$(a,b)_{\mathbb{R}} \cong \mathbb{H}$ if $a<0$ and $b<0$, and 
$(a,b)_{\mathbb{R}} \cong M(2,\mathbb{R})$ if $a>0$ or $b>0$.
(see \cite[Example 4.5]{Conrad}). 
\smallskip

Over the 
complex numbers, all quaternion algebras are split, that is, for alle $a,b\in \mathbb{C}\setminus \{0\}$
we have $(a,b)_{\mathbb{C}}\cong M(2,\mathbb{C})$ (see \cite[Example 4.6]{Conrad}). 
\smallskip

The situation for quaternion algebras over the rational numbers $\mathbb{Q}$ is entirely different.
Actually, there are infinitely many non-isomorphic quaternion algebras $(a,b)_{\mathbb{Q}}$
(see for instance \cite[Corollary 5.5]{Conrad}). 
\end{Remark}

\section{Non-commutative frieze patterns}

We have seen in the previous section that the class of quaternion algebras provides numerous
examples of non-commutative division algebras. These algebraic structures are needed for the 
theory of non-commutative friezes, originating in the work of Berenstein and Retakh
\cite{BR18}. For more details on non-commu\-ta\-tive friezes see also \cite{CHJ24}, \cite{CHJ25}. 
\smallskip

In this section we collect the basic definitions and fundamental results on non-commu\-ta\-tive 
friezes, as far as they are needed in this paper. 

\begin{Definition}[Berenstein, Retakh \cite{BR18}]
\label{def:ncpolygon}
Let $D$ be a ring and 
$\mathcal{P}$ be an $m$-gon (where $m\ge 3$). For each pair of vertices $r,s$ of
$\mathcal{P}$ invertible elements $c_{r,s}\in D^{\ast}$ and $c_{s,r}\in D^{\ast}$ are assigned, giving a map 
$c:\mathrm{diag(\mathcal{P})}\to D^{\ast}$, $c((r,s))=c_{r,s}$, where $\mathrm{diag(\mathcal{P})}$ denotes the set of diagonals
of $\mathcal{P}$.
\begin{enumerate}
\item[{(a)}] The {\em triangle relation} for distinct vertices $r,s,t$ of $\mathcal{P}$ is the equation
\begin{equation} \label{eq:triangle}
c_{r,s}c_{t,s}^{-1}c_{t,r} = c_{r,t}c_{s,t}^{-1}c_{s,r}.
\end{equation}
See Figure \ref{fig:triangle} for an illustration of the triangle relations.
\item[{(b)}] Consider a quadrangle with vertices $r,s,t,u$ in $\mathcal{P}$ such that $(r,t)$ and $(s,u)$
are the diagonals of the quadrangle.
The {\em non-commutative exchange relation} for the diagonal $(t,r)$ is the equation
\begin{equation} \label{eq:ncexchange}
c_{t,r} = c_{t,u}c_{s,u}^{-1}c_{s,r}+ c_{t,s} c_{u,s}^{-1}c_{u,r}
\end{equation}
See Figure \ref{fig:exchange} for an illustration of the exchange relations.
\end{enumerate}
\end{Definition}

\begin{figure}
\begin{center}
\begin{tikzpicture}[auto]
    \node[name=s, draw, shape=regular polygon, regular polygon sides=500, minimum size=3.4cm] {};
    \draw[blue,thick, stealth-] (s.corner 1) to (s.corner 207);
    \draw[red,thick, -stealth] (s.corner 10) to (s.corner 198);
    \draw[blue,thick, -stealth] (s.corner 198) to (s.corner 339);
    \draw[red,thick, stealth-] (s.corner 207) to (s.corner 330);
    \draw[blue,thick, -stealth] (s.corner 1) to (s.corner 339);
    \draw[red,thick, stealth-] (s.corner 10) to (s.corner 330);
    
    \draw[shift=(s.corner 198)]  node[below]  {{\small $s$}};
    \draw[shift=(s.corner 1)]  node[above]  {{\small $r$}};
    \draw[shift=(s.corner 320)]  node[right]  {{\small $t$}};
   \end{tikzpicture}  
\end{center}
\caption{The triangle relation ${\red c_{r,s}c_{t,s}^{-1}c_{t,r}} = {\blue c_{r,t}c_{s,t}^{-1}c_{s,r}}$.\\
(This figure and the following ones are reproduced from our paper
\cite{CHJ24}).}
 \label{fig:triangle}
\end{figure}
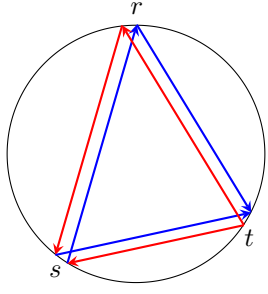

\begin{figure} 
\begin{center}
\begin{tikzpicture}[auto]
    \node[name=s, draw, shape=regular polygon, regular polygon sides=500, minimum size=3.4cm] {};
    
    \draw[blue,thick, stealth-] (s.corner 53) to (s.corner 182);
    
    \draw[red,thick, -stealth] (s.corner 187) to (s.corner 306);
    \draw[blue,thick, stealth-] (s.corner 50) to (s.corner 316);
    \draw[red,thick, -stealth] (s.corner 45) to (s.corner 310);
    \draw[thick, -stealth] (s.corner 185) to (s.corner 410);
    
    \draw[red,thick, -stealth] (s.corner 45) to (s.corner 413);
     \draw[blue,thick, -stealth] (s.corner 318) to (s.corner 406);

    \draw[shift=(s.corner 178)]  node[below]  {{\small $t$}};
    \draw[shift=(s.corner 50)]  node[above]  {{\small $s$}};
    \draw[shift=(s.corner 300)]  node[right]  {{\small $u$}};
    \draw[shift=(s.corner 420)]  node[right]  {{\small $r$}};
   \end{tikzpicture}  
\end{center}
\caption{The non-commutative exchange relation
$c_{t,r} = {\red c_{t,u}c_{s,u}^{-1}c_{s,r}} + {\blue c_{t,s} c_{u,s}^{-1}c_{u,r}}$.
}
\label{fig:exchange}
\end{figure}
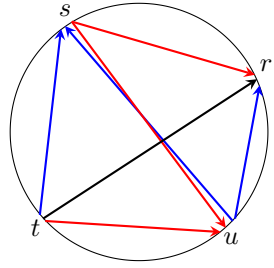

In our paper \cite{CHJ24} we have shown that all triangle relations and all non-commutative exchange 
relations follow from a much smaller set of relations. We recall this here and to make this precise we need the following 
notions from \cite{CHJ24}.

\begin{Definition} \label{def:nclocal}
Let $D$ be a ring and 
$\mathcal{P}$ be an $m$-gon (where $m\ge 3$). Suppose that to each pair of vertices $r,s$ of
$\mathcal{P}$ invertible elements $c_{r,s}\in D^{\ast}$ and $c_{s,r}\in D^{\ast}$ are assigned.
\begin{enumerate}
\item[{(a)}] The {\em local triangle relation} for three consecutive vertices $r,r+1,r+2$ of $\mathcal{P}$ 
is the equation
\begin{equation} \label{eq:localtriangle}
c_{r+1,r+2}c_{r,r+2}^{-1}c_{r,r+1} = c_{r+1,r}c_{r+2,r}^{-1}c_{r+2,r+1}.
\end{equation}
See Figure \ref{fig:localtriangle}.
\item[{(b)}] Consider a quadrangle with vertices $r,r+1,s,s+1$ in $\mathcal{P}$.
The {\em local non-commutative exchange relation} for the diagonal $(s,r)$ is the equation
\begin{equation} \label{eq:nclocalexchange}
c_{s,r} = c_{s,s+1}c_{r+1,s+1}^{-1}c_{r+1,r} + c_{s,r+1} c_{s+1,r+1}^{-1}c_{s+1,r}.
\end{equation}
See Figure \ref{fig:localexchange}.
There are analogous local 
non-commutative exchange relations for the other three diagonals of this quadrangle, namely:
$$
c_{r,s} = c_{r,r+1}c_{s+1,r+1}^{-1}c_{s+1,s} + c_{r,s+1} c_{r+1,s+1}^{-1}c_{r+1,s},
$$
$$
c_{r+1,s+1} = c_{r+1,r}c_{s,r}^{-1}c_{s,s+1} + c_{r+1,s} c_{r,s}^{-1}c_{r,s+1},
$$
$$
c_{s+1,r+1} = c_{s+1,s}c_{r,s}^{-1}c_{r,r+1} + c_{s+1,r} c_{s,r}^{-1}c_{s,r+1}.
$$
\item[{(c)}] A {\em non-commutative frieze} over a ring $D$ on a polygon $\mathcal{P}$ is a map $c:\mathrm{diag(\mathcal{P})}\to D^{\ast}$, $c((r,s))=c_{r,s}$, 
satisfying all local triangle relations and all local non-commutative exchange relations (for all diagonals in the 
quadrangles with vertices $r,r+1,s,s+1$). 
\end{enumerate}
\end{Definition}

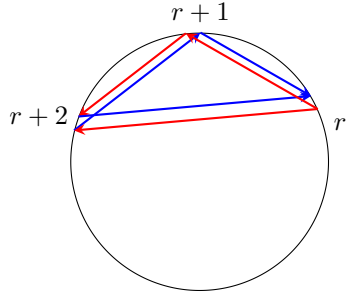
\begin{figure}
\begin{center}
\begin{tikzpicture}[auto]
    \node[name=s, draw, shape=regular polygon, regular polygon sides=500, minimum size=3.4cm] {};
    \draw[blue,thick, stealth-] (s.corner 1) to (s.corner 107);
    \draw[red,thick, -stealth] (s.corner 10) to (s.corner 98);
    \draw[blue,thick, -stealth] (s.corner 98) to (s.corner 419);
    \draw[red,thick, stealth-] (s.corner 107) to (s.corner 410);
    \draw[blue,thick, -stealth] (s.corner 1) to (s.corner 419);
    \draw[red,thick, stealth-] (s.corner 10) to (s.corner 410);
    
    \draw[shift=(s.corner 98)]  node[left]  {{\small $r+2$}};
    \draw[shift=(s.corner 1)]  node[above]  {{\small $r+1$}};
    \draw[shift=(s.corner 400)]  node[right]  {{\small $r$}};
   \end{tikzpicture}  
\end{center}
\caption{The local triangle relations
$
{\red c_{r+1,r+2}c_{r,r+2}^{-1}c_{r,r+1}} = {\blue c_{r+1,r}c_{r+2,r}^{-1}c_{r+2,r+1}}
$.}
\label{fig:localtriangle}
\end{figure}

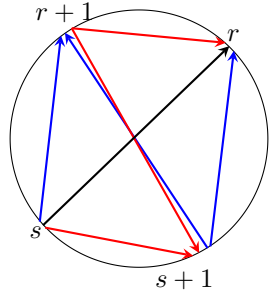
\begin{figure}
\begin{center}
\begin{tikzpicture}[auto]
    \node[name=s, draw, shape=regular polygon, regular polygon sides=500, minimum size=3.4cm] {};
    
    \draw[blue,thick, stealth-] (s.corner 53) to (s.corner 182);
    
    \draw[red,thick, -stealth] (s.corner 187) to (s.corner 286);
    \draw[blue,thick, stealth-] (s.corner 50) to (s.corner 296);
    \draw[red,thick, -stealth] (s.corner 45) to (s.corner 290);
    \draw[thick, -stealth] (s.corner 185) to (s.corner 440);
    
    \draw[red,thick, -stealth] (s.corner 45) to (s.corner 443);
     \draw[blue,thick, -stealth] (s.corner 298) to (s.corner 436);

    \draw[shift=(s.corner 178)]  node[below]  {{\small $s$}};
    \draw[shift=(s.corner 50)]  node[above]  {{\small $r+1$}};
    \draw[shift=(s.corner 280)]  node[below]  {{\small $s+1$}};
    \draw[shift=(s.corner 450)]  node[right]  {{\small $r$}};
   \end{tikzpicture}  
\end{center}
\caption{The local non-commutative exchange relations \\
$c_{s,r} = {\red c_{s,s+1}c_{r+1,s+1}^{-1}c_{r+1,r}} + {\blue c_{s,r+1} c_{s+1,r+1}^{-1}c_{s+1,r}}
$.}
\label{fig:localexchange}
\end{figure}

Note that the definition of a frieze requires only certain triangle relations and non-commutative 
exchange relations to hold. One of the main results of \cite{CHJ24} states
that these local relations imply all triangle and exchange relations. 

\begin{Theorem}[\cite{CHJ24}, Theorem 7.1]
\label{thm:friezepolygon}
Let $c:\mathrm{diag}(\mathcal{P})\to D^{\ast}$ be a non-commutative frieze on a polygon $\mathcal{P}$
over a ring $D$. Then $c$ satisfies all triangle relations and all non-commutative exchange relations.
\end{Theorem}

\begin{Remark}[\cite{CHJ24}, Section 4]
Given a non-commutative frieze $c:\mathrm{diag}(\mathcal{P})\to D^{\ast}$, the values $c((r,s))=c_{r,s}$
can be arranged in a pattern of numbers as in Figure \ref{fig:pattern}. 
This is a non-commutative analogue of a frieze pattern (with coefficients), as defined by Coxeter \cite{Cox71}
and later studied in numerous papers; see the survey \cite{MG15} for an overview and \cite{CHJ20}
for frieze patterns with coefficients.

For such a resulting
{\em non-commutative frieze pattern} we often use the notation $\mathcal{C}=(c_{r,s})$.  
By a slight abuse of notation we also use this in the sequel as a short-hand notation for a non-commutative frieze 
$c:\mathrm{diag}(\mathcal{P})\to D^{\ast}$, $c((r,s))=c_{r,s}$. 
\begin{figure} 
$$\begin{array}{ccccccccccccc}
0~~ & c_{0,1} & c_{0,2} & c_{0,3}~~~ & \ldots & c_{0,m-1} & 0 & & & & & &\\
~~ \\
& 0 & c_{1,2} & c_{1,3} & \ldots & c_{1,m-1} & c_{1,0} & 0 & & & & &\\
~~ \\
& & 0 & c_{2,3} & \ldots & c_{2,m-1} & c_{2,0} & c_{2,1} & ~~~~~0 & & & &\\
~~\\
& &  & \ddots  &  \ddots & \vdots & \vdots &  \vdots & \ddots & \ddots & & & \\
~~\\
& & &  & \ddots & c_{m-2,m-1} & c_{m-2,0} & c_{m-2,1} & \ldots  &  c_{m-2,m-3} & 0 & & \\
~~\\
& & & &  &   0 &  c_{m-1,0} &  c_{m-1,1} & \ldots & \ldots &  
c_{m-1,m-2} & ~~~~~~~~0 &  \\
\end{array}
$$
\caption{A non-commutative frieze pattern.}
 \label{fig:pattern}
\end{figure}
 
\end{Remark}

\begin{Example}[\cite{CHJ24}, Example 4.2] \label{ex:quat}
Here is an example of a non-commutative frieze pattern over the quaternions:
$$
\begin{array}{cccccccccccc}
0 & ~1~ & ~~\mathrm{i}~~ & 1 - \mathrm{k} & -\mathrm{i} - 2\mathrm{j} & 1 & 0 & \sps & \sps & \sps & \sps & \sps \\
    & 0 & 1 & -2\mathrm{i} - \mathrm{j} & 3\mathrm{k} & -\mathrm{i} + \mathrm{j} & 1 & 0 &    &    &    &   \\
    &    & 0 & 1 & \mathrm{i} - \mathrm{j} & \mathrm{k} & \mathrm{i} & 1 & 0 &    &    &   \\
    &    &    & 0 & 1 & \mathrm{j} & 1 + \mathrm{k} & -2\mathrm{i} - \mathrm{j} & 1 & 0 &    &   \\
    &    &    &    & 0 & 1 & -\mathrm{i} - 2\mathrm{j} & -3\mathrm{k} & ~\mathrm{i} - \mathrm{j}~ & 1 & 0 &   \\
\sps & \sps & \sps & \sps & \sps & 0 & 1 & -\mathrm{i} + \mathrm{j} & -\mathrm{k} & ~\mathrm{j}~ & ~1~ & 0
\end{array}
$$
This non-commutative frieze pattern has height 3, i.e.\ it corresponds to a non-commutative frieze on
a hexagon. Note that for all values on diagonals of length 3 we have $c_{r,r+3}\neq c_{r+3,r}$.
\end{Example}

In this article we will mainly be concerned with non-commutative friezes for which all boundary entries are 
equal to 1. Also in the classic situation of friezes this is by far the most commonly studied case, starting with the
classic result of Conway and Coxeter about a bijection between friezes over $\mathbb{N}$ and triangulations
of polygons \cite{CC73}.

\begin{Proposition} \label{rem:equalquddity}
Let $D$ be a ring and let $\mathcal{C}=(c_{r,s})$ be a non-commutative frieze
of height $n\ge 1$ over $D$, such that all boundary entries are equal to 1 (i.e. 
$c_{r,r+1}=c_{r+1,r}$ for all $r$). Then for all $r$ we have $c_{r,r+2}=c_{r+2,r}$. 
\end{Proposition}

\begin{proof}
In fact, we consider the triangle relation for the triangle with vertices
$r,r+1,r+2$, this has the form
$$c_{r,r+2}c_{r+1,r+2}^{-1}c_{r+1,r}=c_{r,r+1}c_{r+2,r+1}^{-1}c_{r+2,r}.
$$
Since all boundary entries are equal to 1 by assumption, this simplifies to
$c_{r,r+2}=c_{r+2,r}.$
\end{proof}

\begin{Example} \label{ex:height1}
Non-commutative friezes of height 1 over a ring $D$ are maps 
$c:\mathrm{diag}(\mathcal{P})\to D^{\ast}$ on the diagonals of a quadrangle $\mathcal{P}$
(see Definition \ref{def:nclocal}\,(c)). Suppose that all boundary values of $c$ are 1.
Labelling the vertices of $\mathcal{P}$ by 0,1,2,3 we have $a:=c_{0,2}=c_{2,0}$ and 
$b:=c_{1,3}=c_{3,1}$ by Proposition \ref{rem:equalquddity}.

The triangle relations are clearly satisfied, the exchange relations have the form
$$ a = c_{0,2} = 1\cdot c_{3,1}^{-1}\cdot 1 + 1\cdot c_{1,3}^{-1}\cdot 1 = 2b^{-1}.
$$
Hence a non-commutative frieze over $D$ on a quadrangle (with all boundary values 1) is given by an ordered pair 
$(a,b)\in D^2$ such that $ab=2$. 
\end{Example}

\begin{Example} \label{ex:ncfriezes-1-2-3}
We give examples of non-commutative friezes (with
all boundary values 1) for
small heights over certain quaternion algebras. 
\begin{enumerate}
\item[{(1)}] In Example \ref{ex:height1} we have 
seen that non-commutative friezes of height 1 are
given by ordered pairs of elements whose product is
2. 

For the quaternion algebra $(-1,b)_{\mathbb{Q}}$
we have $2\mathrm{i}\cdot (-\mathrm{i})=2$,
yielding a non-commutative frieze of height 1
with non-trivial diagonal 
$\ldots, 2\mathrm{i},-\mathrm{i},
2\mathrm{i},-\mathrm{i},2\mathrm{i},-\mathrm{i}
\ldots$.

Similarly, for the quaternion algebra 
$(a,-1)_{\mathbb{Q}}$
we have $2\mathrm{j}\cdot (-\mathrm{j})=2$,
yielding a non-commutative frieze of height 1
with non-trivial diagonal 
$\ldots, 2\mathrm{j},-\mathrm{j},
2\mathrm{j},-\mathrm{j},2\mathrm{j},-\mathrm{j}
$.
\item[{(2)}] For the quaternion algebra $(-2,b)_{\mathbb{Q}}$
we have $\mathrm{i}\cdot (-\mathrm{i})=2$,
yielding a non-commutative frieze of height 1
with non-trivial diagonal 
$\ldots, \mathrm{i},-\mathrm{i},
\mathrm{i},-\mathrm{i},\mathrm{i},-\mathrm{i}
\ldots$.

Similarly, for the quaternion algebra 
$(a,-2)_{\mathbb{Q}}$
we have $\mathrm{j}\cdot (-\mathrm{j})=2$,
yielding a non-commutative frieze of height 1
with non-trivial diagonal 
$\ldots, \mathrm{j},-\mathrm{j},
\mathrm{j},-\mathrm{j},\mathrm{j},-\mathrm{j}
$.
\item[{(3)}] For the quaternion algebra
$(-3,b)_{\mathbb{Q}}$ we have 
$\mathrm{i}\cdot (-\mathrm{i})=3$. This 
leads to a non-commutative frieze of height 3
by setting
$$c_{r+2,r}=c_{r,r+2} = \left\{
\begin{array}{cl} \mathrm{i} & \text{ if $r$ is even} \\
-\mathrm{i} & \text{ if $r$ is odd}
\end{array} \right. 
$$
and $c_{r+3,r}=c_{r,r+3}=2$ for all $r$. 
It is straightforward to check that all local 
triangle and exchange relations are satisfied. 
The corresponding non-commutative frieze pattern has the 
form
$$\begin{array}{cccccccccccc}
0 & 1 & \mathrm{i} & 2 & \mathrm{i} & 1 & 0 & & & & & \\
& 0 & 1 & -\mathrm{i} & 2 & -\mathrm{i} & 1 & 0 & & & & \\
& & 0 & 1 & \mathrm{i} & 2 & \mathrm{i} & 1 & 0 & & & \\
& & & 0 & 1 & -\mathrm{i} & 2 & -\mathrm{i} & 1 & 0 & & \\
& & & & 0 & 1 & \mathrm{i} & 2 & \mathrm{i} & 1 & 0 & \\
& & & & & 0 & 1 & -\mathrm{i} & 2 & -\mathrm{i} & 1 & 0
\end{array}
$$
\end{enumerate}
Similarly, for the quaternion algebra 
$(a,-3)_{\mathbb{Q}}$ we have a 
non-commutative frieze of height 3 with 
$\mathrm{j}$ instead of $\mathrm{i}$.
\end{Example}
\smallskip

A crucial feature of non-commutative friezes is that for any pair of vertices $r,s$ of the polygon we have 
two values $c_{r,s}$ and $c_{s,r}$. In the non-commutative frieze pattern over the quaternions $\mathbb{H}$
in Example \ref{ex:quat} we have seen that these values might be different (even if all boundary values are 1). 
However, in the above Example \ref{ex:quat}, the values $c_{r,r+3}$ and $c_{r+3,r}$ still have the same norm.
A priori, this looked like an artefact of this example, but,  
surprisingly, it is true for all non-commutative friezes with 1's on the boundary over arbitrary normed division rings.

\begin{Lemma} \label{lem:samenorm}
Let $D$ be a normed division ring and let $\mathcal{C}=(c_{r,s})$ be a non-commutative
frieze
of height $n\ge 1$ over $D\setminus \{0\}$, such that all boundary entries are equal to 1. Then for all $r,s$ we have
$\lVert c_{r,s}\rVert = \lVert c_{s,r}\rVert$. 
\end{Lemma}

\begin{proof}
We fix $r$ and proceed by induction on $\ell$ to show that 
$\lVert c_{r,r+\ell}\rVert = \lVert c_{r+\ell,r}\rVert$ for all 
$\ell\ge 0$ (i.e. we induct on the distance between the vertices). 
This is clear for $\ell=0$, it holds for $\ell=1$, since 
$c_{r,r+1}=1=c_{r+1,r}$
by assumption, and it holds for $\ell=2$ by Proposition
\ref{rem:equalquddity}.  

Assume as induction hypothesis, that for all vertices $s$ we have
$\lVert c_{s,s+\ell}\rVert=\lVert c_{s+\ell,s}\rVert$ for some $\ell$ (i.e. the claim holds for any vertices with distance $\ell$). The triangle relation
for the triangle with vertices $r,r+1,r+\ell+1$ reads
$$ c_{r,r+\ell+1}c_{r+1,r+\ell+1}^{-1}c_{r+1,r}=c_{r,r+1}c_{r+\ell+1,r+1}^{-1}c_{r+\ell+1,r}.
$$
Using the assumption that the boundary entries are 1 and taking norms we obtain
$$\lVert c_{r,r+\ell+1}\rVert \cdot \lVert c_{r+1,r+\ell+1}\rVert^{-1} = \lVert c_{r+\ell+1,r+1}\rVert^{-1}
\cdot \lVert c_{r+\ell+1,r}\rVert.
$$
We can apply the induction hypothesis to $c_{r+1,r+\ell+1}$, as the indices have distance $\ell$, hence
$\lVert c_{r+1,r+\ell+1}\rVert^{-1} = \lVert c_{r+\ell+1,r+1}\rVert^{-1}$ and cancelling this in the previous equation yields
$$\lVert c_{r,r+\ell+1}\rVert = \lVert c_{r+\ell+1,r}\rVert$$
which completes the induction step. 
\end{proof}
\medskip

\section{Finiteness for non-commutative friezes over normed division rings} \label{sec:finiteness}

A fundamental question when studying friezes over a given set of numbers is whether there are finitely or infinitely 
many friezes with entries in this set, for any given height. The following result gives an upper bound
on the norms of the entries of the quiddity sequence of a non-commutative frieze over a normed
division ring. It generalizes \cite[Lemma 6.1]{CHJ20} from 
classic friezes with coefficients over complex numbers 
to non-commutative friezes with coefficients over 
normed division rings.
Actually, the proof of \cite[Lemma 6.1]{CHJ20}
can largely be transferred to the non-commutative setting.
However, we decided to present here a simpler proof
leading to a better bound
which is based on discussions
with ChatGPT 6.0 Pro on this result for commutative friezes; 
we then generalized this approach to non-commutative friezes, leading to the following result.

\begin{Theorem} \label{lem:finite}
Let $D$ be a normed division ring. 
Let $R\subseteq D \setminus\{0\}$ be a subset such that 
$$M:=\inf\{\lVert x\rVert\,:\,x\in R\}>0.
$$
Let $\mathcal{C}=(c_{i,j})$ be a non-commutative frieze 
on an $(n+3)$-gon with 
$n\ge 1$
over $R$ and set
$$P= \max \{ \lVert c_{i,i+1}\rVert, \lVert c_{i+1,i}\rVert \,:\, 0\le i\le n+2\},
$$
the maximal value of the norm of a boundary entry. 
Then the norm of every quiddity entry $c_{j,j+2}$ and $c_{j+2,j}$ of $\mathcal{C}$ is 
at most 
$$\frac{P^2(M+nP)}{M^2}.
$$
\end{Theorem}

\smallskip

\begin{proof}
We denote the vertices of the $(n+3)$-gon by $0,1,\ldots,n+2$. By symmetry it suffices to consider the
quiddity entry $c_{1,n+2}$. For every $i=1,\ldots,n$
the non-commutative exchange 
relation for the quadrangle with vertices $0,i,i+1,n+2$
yields
$$c_{0,i+1} = c_{0,i}c_{n+2,i}^{-1}c_{n+2,i+1} + 
c_{0,n+2}c_{i,n+2}^{-1} c_{i,i+1}.
$$
By multiplication with $c_{i,i+1}^{-1}c_{i,n+2}$ from the
right and then using the triangle relation for the triangle 
with vertices $i,i+1,n+2$ we can deduce that
\begin{eqnarray*}
c_{0,i+1}c_{i,i+1}^{-1}c_{i,n+2} & = & 
c_{0,i}c_{n+2,i}^{-1}c_{n+2,i+1}c_{i,i+1}^{-1}c_{i,n+2} + 
c_{0,n+2} \\
& = & c_{0,i}c_{n+2,i}^{-1} c_{n+2,i}c_{i+1,i}^{-1} c_{i+1,n+2} + c_{0,n+2} \\
& = & c_{0,i}c_{i+1,i}^{-1} c_{i+1,n+2} + c_{0,n+2}
\end{eqnarray*}
We multiply this equation with $c_{0,i}^{-1}$ from the left
to get
\begin{equation} \label{eq:triangle1}
c_{0,i}^{-1}c_{0,i+1}c_{i,i+1}^{-1}c_{i,n+2} =
c_{i+1,i}^{-1} c_{i+1,n+2} + c_{0,i}^{-1}c_{0,n+2}.
\end{equation}
The triangle relation for the triangle with vertices
$i,i+1,0$ reads $c_{i,i+1}c_{0,i+1}^{-1}c_{0,i}
= c_{i,0}c_{i+1,0}^{-1}c_{i+1,i}$. Inverting this and
plugging it in in (\ref{eq:triangle1}) yields
$$c_{i+1,i}^{-1}c_{i+1,0}c_{i,0}^{-1} c_{i,n+2} = 
c_{i+1,i}^{-1} c_{i+1,n+2} + c_{0,i}^{-1}c_{0,n+2}.
$$
We multiply this equation with $c_{i+1,0}^{-1}c_{i+1,i}$
from the left and get
$$c_{i,0}^{-1} c_{i,n+2} -
c_{i+1,0}^{-1} c_{i+1,n+2} = 
c_{i+1,0}^{-1}c_{i+1,i}c_{0,i}^{-1}c_{0,n+2}.
$$
We sum these equations for $i=1,\ldots,n$ and observe
that on the left hand side a telescope sum appears 
where only the first and last entry remains. So we obtain
$$c_{1,0}^{-1}c_{1,n+2} - c_{n+1,0}^{-1}c_{n+1,n+2}
= \sum_{i=1}^n c_{i+1,0}^{-1}c_{i+1,i}c_{0,i}^{-1}
c_{0,n+2}.
$$
By multiplication with $c_{1,0}$ from the left we 
get an expression for the quiddity entry $c_{1,n+2}$
under consideration:
\begin{equation} \label{eq:c1n+2}
c_{1,n+2} = c_{1,0}c_{n+1,0}^{-1}c_{n+1,n+2}
+ c_{1,0}\left(\sum_{i=1}^n c_{i+1,0}^{-1}c_{i+1,i}c_{0,i}^{-1}\right)
c_{0,n+2}.
\end{equation}
By definition of $M$ and $P$ we know that the norm of
each boundary entry is at most $P$ and that for each 
entry of the frieze we have $\lVert c_{r,s}\rVert\ge M$,
hence $\lVert c_{r,s}^{-1}\rVert \le \frac{1}{M}$. 
Using this and applying the triangle inequality to
(\ref{eq:c1n+2}) we obtain
\begin{eqnarray*}
\lVert c_{1,n+2}\rVert & \le & 
\lVert c_{1,0}\rVert\cdot \lVert c_{n+1,0}^{-1}\rVert \cdot \lVert c_{n+1,n+2}\rVert
+ \lVert c_{1,0}\rVert \left(\sum_{i=1}^n \lVert c_{i+1,0}^{-1}\rVert \cdot \lVert c_{i+1,i}\rVert\cdot \lVert c_{0,i}^{-1}\rVert \right)
\lVert c_{0,n+2}\rVert \\
& \le & \frac{P^2}{M} + P\cdot 
\left(\sum_{i=1}^n \frac{P}{M^2}\right)\cdot P 
\,=\, \frac{P^2(M+nP)}{M^2} 
\end{eqnarray*}
as claimed in the theorem.
\end{proof}

\smallskip

As a main application of Theorem \ref{lem:finite} we can show for many subsets $R$ of a normed division 
ring that there are only finitely many non-commutative friezes over $R\setminus \{0\}$, for any given height. We shall need the 
following definition.

\begin{Definition} \label{def:normedsubset}
Let $D$ be a normed division ring. 
\begin{enumerate}
\item[{(i)}] A subset $R\subseteq D\setminus \{0\}$ is called {\em norm-discrete} if the set
$\{\lVert a\rVert\,|\,a\in R\}$ is a discrete subset of $\mathbb{R}_{\ge 0}$ (i.e. it has no accumulation
point with the usual topology). 
\item[{(ii)}] A subset $R\subseteq D\setminus \{0\}$ is called {\em norm-finite} if 
for every real number $N$ the set
$$\{a\in R\,|\, \lVert a\rVert\le N\}$$
is finite.  
\end{enumerate}
\end{Definition}

Note that any norm-finite subset is
norm-discrete. Moreover, every subset of a norm-finite subset is also norm-finite. 

\begin{Example} \label{ex:lipschitz}
In the normed division ring $\mathbb{H}$ of Hamilton quaternions the {\em Lipschitz quaternions}
$$\mathbb{H}_{\mathbb{Z}}:= \{\alpha +  \beta \mathrm{i}+\gamma\mathrm{j}+\delta\mathrm{k}\,|\,\alpha,\beta ,\gamma,\delta\in \mathbb{Z}\}
$$
and the {\em Hurwitz quaternions}
$$\mathbb{H}_{\frac{1}{2}\mathbb{Z}} :=
\{\alpha +\beta \mathrm{i}+\gamma\mathrm{j}+ \delta \mathrm{k}\,|\,\alpha,\beta,\gamma,\delta\in \mathbb{Z}\mbox{~or~}
\alpha,\beta,\gamma,\delta\in \frac{1}{2}+\mathbb{Z}\}
$$
are norm-finite subsets.  
\smallskip

For any integer $d\ge 2$, the subset 
$$\mathbb{H}_{\mathbb{Z}(\sqrt{d})} 
= \{\alpha +\beta\mathrm{i}+\gamma\mathrm{j}+\delta\mathrm{k}\,|\,\alpha,\beta,\gamma,\delta\in \mathbb{Z}(\sqrt{d})\}
$$
is not norm-discrete (and hence not norm-finite) since 
$\mathbb{Z}(\sqrt{d})\subseteq \mathbb{H}_{\mathbb{Z}(\sqrt{d})}$
contains elements with arbitrarily small norm. 
\end{Example}

\smallskip

We can then state our first application of Theorem \ref{lem:finite}, which generalizes \cite[Corollary 3.8]{CH19}.

\begin{Corollary} \label{cor:finitelymany}
Let $D$ be a normed division ring and $R\subseteq D\setminus \{0\}$ a norm-finite subset. 
Then for each $n\in \mathbb{N}$ there are only finitely many non-commutative friezes over $R$ of height $n$
with all boundary entries 1. 
\end{Corollary}

\begin{proof}
The set $R$ is norm-finite, in particular it is norm-discrete, so 
the infimum $M=\inf\{\lVert x\rVert\,:\,x\in R\}$ is positive and the assumption of Theorem \ref{lem:finite}
is satisfied for $R$. Then Theorem \ref{lem:finite} implies that there is an upper bound for the quiddity entries
$c_{j,j+2}$ and $c_{j+2,j}$ of any non-commutative frieze $\mathcal{C}$ of height $n$. Since $R$ is 
norm-finite, there are only finitely many possible values for the quiddity entries. But every non-commutative
frieze is uniquely determined by its quiddity entries.
\end{proof}

\begin{Example}
\begin{enumerate}
\item[{(1)}]
There are precisely 40 non-commutative friezes of height 1 over the Lipschitz quaternions. 
 
The non-commutative friezes of height 1 are parametrized by ordered pairs
$(\alpha,\beta)\in \mathbb{H}_{\mathbb{Z}}$ such that $\alpha\beta=2$ (see Example \ref{ex:height1}). 
Taking norms 
this condition implies $\lVert \alpha\rVert\cdot \lVert \beta\rVert=2$  and from this it is straightfoward to produce a list
of such pairs. In total, there are 40 such pairs, namely the following 20 pairs and their 
shifts obtained by interchanging 
the two entries.
$$(1,2), (-1,-2), (\mathrm{i},-2\mathrm{i}), (-\mathrm{i},2\mathrm{i}),
(\mathrm{j},-2\mathrm{j}), (-\mathrm{j},2\mathrm{j}), 
(\mathrm{k},-2\mathrm{k}), (-\mathrm{k},2\mathrm{k}),
$$
$$
(1+\mathrm{i},1-\mathrm{i}), (-1-\mathrm{i},-1+\mathrm{i}),
(1+\mathrm{j},1-\mathrm{j}), (-1-\mathrm{j},-1+\mathrm{j}),
(1+\mathrm{k},1-\mathrm{k}), (-1-\mathrm{k},-1+\mathrm{k}),
$$
$$
(\mathrm{i}+\mathrm{j},-\mathrm{i}-\mathrm{j}),
(\mathrm{i}-\mathrm{j},-\mathrm{i}+\mathrm{j}),
(\mathrm{i}+\mathrm{k},-\mathrm{i}-\mathrm{k}),
(\mathrm{i}-\mathrm{k},-\mathrm{i}+\mathrm{k}),
(\mathrm{j}+\mathrm{k},-\mathrm{j}-\mathrm{k}),
(\mathrm{j}-\mathrm{k},-\mathrm{j}+\mathrm{k})
$$
\item[{(2)}] There are precisely 72 non-commutative friezes of height 1 over the Hurwitz quaternions. 

In addition to the 40 friezes from (1) we get the 32 friezes for the following 16 pairs and their shifts.
$$\left(\frac{1}{2} (\pm 1\pm \mathrm{i}\pm\mathrm{j}\pm \mathrm{k}),
\pm 1 \mp \mathrm{i}\mp \mathrm{j} \mp \mathrm{k}\right)
$$
\end{enumerate}
\end{Example}

\medskip

Our next goal is to show that for any non-commutative frieze with 1's on the boundary over a normed division ring there must be 
quiddity entries with small norm. We shall need the following matrices which play a crucial role in the theory
of friezes. 

\begin{Definition}
Let $R$ be a ring. For $c\in R$ we define the $\mu$-matrix as
$$\mu(c) := \begin{pmatrix} 0 & -1 \\ 1 & c \end{pmatrix}.
$$
\end{Definition}

For non-commutative friezes the following property of $\mu$-matrices has been shown in an earlier paper. 

\begin{Lemma}[\cite{CHJ24}, Corollary 6.8] \label{lem:etamatrices}
Let $\mathcal{C}=(c_{i,j})$ be a non-commutative frieze on an $m$-gon over some ring $R$ with all
boundary values equal to 1. Then
$$\prod_{r=1}^{m} \mu(c_{r-1,r+1}) = \begin{pmatrix} -1 & 0 \\ 0 & -1 \end{pmatrix}.
$$ 
\end{Lemma}

For proving that the quiddity sequence of any non-commutative frieze pattern contains some entries
with small norm we shall need the following non-commutative version 
of \cite[Lemma 3.1]{CH19}.

\begin{Lemma} \label{lem:cde}
Let $D$ be a normed division ring and for some integer $m\ge 3$ let $c_1,\ldots,c_m,d,e\in D$ with $\lVert c_1\rVert\ge 1$ and 
$\lVert dc_m-e\rVert > \lVert d\rVert$.  Moreover, suppose that 
$$\prod_{r=1}^m \mu(c_r) = \begin{pmatrix} \ast & \ast\\ d & e \end{pmatrix}
$$
(where each $\ast$ denotes an arbitrary entry). Then there exists an index
$j\in \{2,\ldots,m-1\}$ such that $\lVert c_j\rVert <2$.  
\end{Lemma}

\begin{proof}
Let $a,b\in D$ with $\lVert a\rVert \le \lVert b\rVert$ and 
$\lVert c\rVert \ge 2$. Then the reverse triangle inequality gives
\[ \lVert bc-a\rVert \ge \lVert bc\rVert -\lVert a\rVert = \lVert b\rVert (\lVert c\rVert -1)+\lVert b\rVert - \lVert a\rVert  \ge \lVert b\rVert (\lVert c\rVert -1)\ge \lVert b\rVert. \]
From this inequality and
\begin{equation} \label{eq:eta}
(a,b)\,\mu(c) = 
(a,b)\begin{pmatrix} 0 & -1 \\ 1 & c \end{pmatrix}  =
\begin{pmatrix} b \\ -a+bc \end{pmatrix}, 
\end{equation}
we see that multiplying row vectors in $D^2$ from the right with $\mu(c)$, where $\lVert c\rVert\ge 2$,
preserves the property that the norm of the first entry is less than or equal to the norm of the second entry.

Now
\[ (0,1)\mu(c_1) = 
(0,1) \begin{pmatrix} 0 & -1 \\ 1 & c_1 \end{pmatrix} = (1,c_1).
\]
From this and the assumption on the shape of $\prod_{r=1}^m \mu(c_r)$ we get
\begin{eqnarray*}
(1,c_1)\left(\prod_{j=2}^{m-1} \mu(c_j)\right)
= (0,1) \begin{pmatrix} \ast & \ast \\ d & e \end{pmatrix}\mu(c_m)^{-1}
=  (d,e)\begin{pmatrix} c_m & 1 \\ -1 & 0 \end{pmatrix} = (dc_m-e,d).
\end{eqnarray*}

Note that by assumption we have $\lVert c_1\rVert \ge 1$ and $\lVert dc_m-e\rVert > \lVert d\rVert$.
Thus $\lVert c_2\rVert,\ldots,\lVert c_{m-1}\rVert$ all greater or equal to $2$ would contradict
the property stated after Equation (\ref{eq:eta}).
\end{proof}

In a special case of Lemma \ref{lem:cde} we can even get a stronger conclusion, which will be
a useful tool for applications on non-commutative friezes. This is a non-commutative version
of \cite[Corollary 3.3]{CH19}.

\begin{Corollary} \label{cor:twonormless2}
Let $D$ be a normed division ring and let $(c_1,\ldots,c_m)\in D^m$ such that
$\prod_{r=1}^m \mu(c_r)$ is a non-zero scalar multiple of the identity matrix. Then there are two different
indices $j,k\in \{1,\ldots,m\}$ with $\lVert c_j\rVert <2$ and 
$\lVert c_k\rVert<2$. 
\end{Corollary}

\begin{proof}
Note that the assumption of the corollary can not be satisfied for $m=1$ and for $m=2$ it only holds
for $c_1=0=c_2$, in which case the conclusion is clearly true. So we can assume now that $m\ge 3$. 

Since any scalar multiple of the identity matrix commutes with every matrix, the assumption on the product of 
$\mu$-matrices also holds for all rotated shifts of the sequence $(c_1,\ldots,c_m)$, in particular it holds
for $(c_2,\ldots,c_{m},c_1)$. 

We distinguish cases. Suppose first that $\lVert c_1\rVert <1$. If also $\lVert c_{2}\rVert<1$, then we are done. 
If $\lVert c_{2}\rVert\ge 1$, then we consider the rotated sequence $(c_{2},\ldots,c_{m},c_1)$. 
It satisfies the assumptions of Lemma \ref{lem:cde} with $d=0$ ($e\neq 0$ by the assumption that 
the product of $\mu$-matrices is a non-zero multiple of the identity matrix). By Lemma \ref{lem:cde}
there exists an index $j\in \{3,\ldots,m\}$ with $\lVert c_j\rVert<2$, and we are done. 

Suppose now that $\lVert c_1\rVert \ge 1$. Again by Lemma \ref{lem:cde} there exists an index 
$j\in \{2,\ldots,m-1\}$ such that $\lVert c_j\rVert <2$. Then we consider the rotated sequence
$$(c_{j+1},c_{j+2},\ldots,c_m,c_1,\ldots,c_{j-1},c_{j}).
$$
If $\lVert c_{j+1}\rVert <1$, then we are done. If $\lVert c_{j+1}\rVert \ge 1$, then the rotated sequence
satisfies the assumptions of Lemma \ref{lem:cde} (with $d=0$). So there exists an index
$k\in \{j+2,\ldots,m,1,\ldots,j-1\}$ such that $\lVert c_k\rVert <2$, and we are done. 
This completes the proof of the corollary.
\end{proof}

We now want to apply the previous results to non-commutative friezes. 

\begin{Corollary} \label{thm:quidditynorm2}
Let $D$ be a normed division ring and $R\subseteq D \setminus\{0\}$ a subset. 
Let $\mathcal{C}=(c_{i,j})$ be a non-commutative frieze on an $m$-gon (where $m\ge 3$) over 
$R$ such that all boundary values are equal to 1.
Then the set
$$\{c_{0,2},c_{1,3},\ldots,c_{m-3,m-1},c_{m-2,0},c_{m-1,1}\}
$$
of quiddity entries contains at least two elements with norm less than 2. 
\end{Corollary}

Recall from Remark \ref{rem:equalquddity} that for the quiddity entries we have $c_{i,i+2}=c_{i+2,i}$
since the boundary entries are all equal to 1, so one could state Corollary \ref{thm:quidditynorm2}
as well for the set 
$$\{c_{2,0},c_{3,1},\ldots,c_{m-1,m-3},c_{0,m-2},c_{1,m-1}\}.$$

\begin{proof}
We know from Lemma \ref{lem:etamatrices} that
$$\prod_{r=1}^{m} \mu(c_{r-1,r+1}) = \begin{pmatrix} -1 & 0 \\ 0 & -1 \end{pmatrix}
$$ 
is a non-zero scalar multiplie of the identity matrix. Then Corollary \ref{cor:twonormless2}
gives that there are two quiddity entries $c_{j-1,j+1}$ and $c_{k-1,k+1}$ with norm less than 2. 
\end{proof}

\section{Non-commutative friezes over Lipschitz subrings of quaternion algebras}

We have seen in the previous section that for norm-finite subsets of normed division rings there 
are only finitely many non-commutative friezes for each height.
In this section we want to apply this to certain subsets of quaternion algebras. 
For using the norm form on a quaternion algebra to define a norm function (in the sense of
Definition \ref{def:normeddivring}) we require the norm form to have values in $\mathbb{R}$
(and not in an arbitrary field). 
Therefore, we consider in this section quaternion algebras over the rational numbers and the real numbers. 
\medskip

Let $a,b\in \mathbb{R}\setminus \{0\}$. The quaternion algebra $(a,b)_{\mathbb{R}}$ 
(and similarly $(a,b)_{\mathbb{Q}}$ if $a,b\in \mathbb{Q}$) carries 
a norm form
$$N:(a,b)_{\mathbb{R}}\to \mathbb{R},~N(\alpha+\beta \mathrm{i}+\gamma \mathrm{j} +
\delta \mathrm{k}) = \alpha^2 - a\beta^2 - b\gamma^2 +ab \delta^2
$$
(see Definition \ref{def:norm}).
To mimic the norm on Hamilton's quaternions the natural approach for a norm function is to set
$$\lVert .\rVert: (a,b)_{\mathbb{R}}\to \mathbb{R},~\lVert q\rVert := \sqrt{ N(q)}
$$
where $q=\alpha+\beta \mathrm{i}+\gamma \mathrm{j} +\delta \mathrm{k}\in (a,b)_{\mathbb{R}}$. 
However, for this to make sense we need the values $N(q)$ to be 
non-negative real numbers. 
This is only guaranteed when we assume $a<0$ and $b<0$. 
\medskip

\begin{Proposition} \label{prop:normab<0}
Let $a<0$ and $b<0$ be real numbers. 
With the above notation, the quaternion algebra $(a,b)_{\mathbb{R}}$ (and also $(a,b)_{\mathbb{Q}}$ for
$a,b\in \mathbb{Q}$) becomes a normed division ring
with the norm function $\lVert .\rVert =\sqrt{N(.)}$. That is (cf. Definition \ref{def:normeddivring}), the following properties are satisfied
for all $q,q_1,q_2\in (a,b)_{\mathbb{R}}$:
\begin{enumerate}
\item[{(i)}] $\lVert q\rVert =0$ $\Longleftrightarrow$ $q=0$.
\item[{(ii)}] $\lVert q_1q_2\rVert  = \lVert q_1\rVert \cdot \lVert q_2\rVert$.
\item[{(iii)}] $\lVert q_1+q_2\rVert \le \lVert q_1\rVert  + \lVert q_2\rVert$.
\end{enumerate}
\end{Proposition}

\begin{proof}
\begin{enumerate}
\item[{(i)}] Since $a<0$ and $b<0$ the value $N(q)$ is a sum of positive multiples of squares, hence
we have $\lVert q\rVert =0$ if and only if $N(q)=0$ if and only if $q=0$. 
\item[{(ii)}] This follows from Proposition \ref{prop:conjnormprop}\,(iii). 
\item[{(iii)}] Since $a<0$ and $b<0$, the function $\lVert .\rVert =\sqrt{N(.)}$ is induced from an inner product
so the triangle inequality is satisfied for $\lVert .\rVert$ (this is a standard application of the Cauchy-Schwarz inequality). 
\end{enumerate}
\end{proof}

\begin{Remark} \label{rem:squares}
In the sequel we will restrict our attention to quaternion algebras over $\mathbb{Q}$. Given a
quaternion algebra $(a,b)_{\mathbb{Q}}$ with $a,b\in \mathbb{Q}$ we can assume, up to isomorphism, 
that $a,b\in \mathbb{Z}$. 
In fact, this follows from the general observation that $(a,b)_F\cong (ac^2,bd^2)_F$ for all non-zero $c,d\in F$
and any field $F$ (see \cite{Conrad}, Remark after Definition 4.1). 
\end{Remark}

\smallskip

So from now on we consider quaternion algebras $(a,b)_{\mathbb{Q}}$ over the rational numbers where
$a$ and $b$ are non-zero integers. 
\medskip

One can immediately generalize the well-known Lipschitz quaternions (see Example \ref{ex:lipschitz})
from Hamilton's quaternions to other quaternion algebras.

\begin{Definition} \label{def:lipschitz}
Let $a,b\in \mathbb{Z}\setminus \{0\}$. The {\em Lipschitz subring} of the quaternion algebra $(a,b)_{\mathbb{Q}}$
is the subring
$$(a,b)_{\mathbb{Z}} = \{\alpha+ \beta \mathrm{i} + \gamma \mathrm{j} + 
\delta \mathrm{k}\,|\,\alpha,\beta,\gamma,\delta\in \mathbb{Z}\}.
$$
\end{Definition}

Note that for $a=b=-1$ the subring $(-1,-1)_{\mathbb{Z}}=\mathbb{H}_{\mathbb{Z}}$ is the 
set of Lipschitz quaternions. 

\medskip

Our aim is to consider non-commutative friezes over quaternion algebras (with all boundary values equal to 1). 
In particular, we want to address the
fundamental question whether for given subsets there are finitely or infinitely many such non-commutative friezes
for any given height. One important aspect for this is to know how many units the given subset contains. If a subset 
contains infinitely many units then there are already infinitely many non-commutative friezes of height 1,
by Example \ref{ex:height1}.
\smallskip

For Lipschitz subrings in quaternion algebras over $\mathbb{Q}$ the question of finitely or infinitely many units is answered by the following result.

\begin{Proposition}
Let $a,b\in \mathbb{Z}\setminus \{0\}$. 
For $a<0$ and $b<0$ the ring $(a,b)_{\mathbb{Z}}$ has finitely many units. On the other hand,  
if $a>0$ or $b>0$, the ring $(a,b)_{\mathbb{Z}}$ has infinitely many units.  
\end{Proposition}
 
\begin{proof}
By Proposition \ref{prop:normab<0} there is a norm 
form $N:(a,b)_{\mathbb{Z}}\to \mathbb{R}$ which takes values in the set of integers. 
Note that because the norm form $N$ is multiplicative (by Proposition \ref{prop:normab<0}\,(ii)), 
an element $q\in (a,b)_{\mathbb{Z}}$ is invertible if
and only if $N(q)=\pm 1$. 
\smallskip

First we consider the case $a<0$ and $b<0$. Then the values of the norm form $N$ are
non-negative integers. 
Hence for every unit $q\in (a,b)_{\mathbb{Z}}$ we have $N(q)=1$. The values of the norm form are
sums of non-negative multiples of squares of integers, so there can only be finitely many $q\in (a,b)_{\mathbb{Z}}$ with
$N(q)=1$. (Actually, for $a=b=-1$ there are eight units $\pm 1,\pm\mathrm{i},\pm\mathrm{j},
\pm \mathrm{k}$, for $a=-1$, $b<-1$ there are four units $\pm 1, \pm\mathrm{i}$, 
similarly for $a<-1$, $b=-1$ there are four units $\pm 1, \pm\mathrm{j}$, and if $a<-1$, $b<-1$ there
are only two units $\pm 1$.) 
\smallskip

Now suppose that $a>0$ or $b>0$, say $b>0$ (the other case is completely analogous).
If $b$ is a perfect square, then by Remark \ref{rem:squares} we have 
$(a,b)_{\mathbb{Q}} \cong (a,1)$. The latter is isomorphic to the matrix ring $M(2,\mathbb{Q})$,
see Remark \ref{rem:a1} and Theorem \ref{thm:notdivision}. The Lipschitz subring of $M(2,\mathbb{Q})$
is the ring $M(2,\mathbb{Z})$ which has infinitely many units. 

Now assume that $b$ is not a perfect square. 
By the introductory remark, units in $(a,b)_{\mathbb{Z}}$
correspond to integral solutions of the equations
\begin{equation} \label{eq:pell}
\alpha^2-a\beta^2-b\gamma^2+ab\delta^2 =\pm 1.
\end{equation}
Since $b$ is not a square the theory of Pell's equation (and Dirichlet's unit theorem) gives infinitely many 
integral solutions to the equations $\alpha^2-b\gamma^2=\pm 1$. Setting $\beta=0$ and $\delta=0$
these also yield infinitely many integral solutions of (\ref{eq:pell}). Hence there are infinitely many 
units in $(a,b)_{\mathbb{Z}}$. 
\end{proof} 
 
 \medskip
 
 For the case $a<0$ and $b<0$ we will show that for each height there are only 
 finitely many non-commutative friezes over $(a,b)_{\mathbb{Z}}\setminus \{0\}$. For this
 we need the following observation.

\begin{Proposition} \label{prop:abZfinite}
If $a<0$ and $b<0$ are integers, the Lipschitz subring $(a,b)_{\mathbb{Z}}$ is a norm-finite subset of the
normed division ring $(a,b)_{\mathbb{Q}}$. 
\end{Proposition}

\begin{proof}
Recall that the norm on $(a,b)_{\mathbb{Q}}$ is given by
$$\lVert \alpha + \beta\mathrm{i} + \gamma \mathrm{j} + \delta\mathrm{k}\rVert =
\sqrt{\alpha^2-a\beta^2-b\gamma^2+ab\delta^2}
$$
and when restricted to the Lipschitz subring $(a,b)_{\mathbb{Z}}$ the values of the norm function 
are roots of non-negative integers (since the integers $a<0$ and $b<0$ by assumption). 
In particular, the set $\{\lVert q\rVert \,:\,q\in (a,b)_{\mathbb{Z}}\}$ is a discrete subset of $\mathbb{R}_{\ge 0}$,
so $(a,b)_{\mathbb{Z}}$ is norm-discrete (see Definition \ref{def:normedsubset}). Moreover, for each
$r\in \mathbb{Z}_{\ge 0}$ there are only finitely many integral solutions to the equation 
$\alpha^2-a\beta^2-b\gamma^2+ab\delta^2=r$ (again use that $a<0$ and $b<0$ are integers).
Hence, for each $N\in \mathbb{R}$ the set $\{q\in (a,b)_{\mathbb{Z}}\,:\,\lVert q\rVert\le N\}$ is finite.
By Definition \ref{def:normedsubset} this means that the Lipschitz subring $(a,b)_{\mathbb{Z}}$ is norm-finite. 
\end{proof}

\begin{Corollary}
Let $a<0$ and $b<0$ be integers. For each height there are only finitely many 
non-commutative friezes over $(a,b)_{\mathbb{Z}}\setminus \{0\}$. 
\end{Corollary}

\begin{proof}
We know from Proposition \ref{prop:abZfinite} that $(a,b)_{\mathbb{Z}}$ (and hence also 
$(a,b)_{\mathbb{Z}}\setminus \{0\}$) is a norm-finite subset of the normed division ring $(a,b)_{\mathbb{Q}}$.  
Then the claim follows from Corollary \ref{cor:finitelymany}.
\end{proof}

We extend a definition from \cite{CHP24} to the non-commutative setting. 

\begin{Definition} \label{def:friezesubring}
Let $R$ be a ring. The {\em non-commutative frieze subring} of $R$ is defined as the ring $R^{\circ}$ 
generated by all entries of all non-commutative frieze patterns over $R$ with 1's on the boundary. 
\end{Definition}

\begin{Theorem} \label{thm:abcirc}
Let $a<0$ and $b<0$ be integers. For the non-commutative frieze subring of the Lipschitz
quaternion ring we get 
$$(a,b)_{\mathbb{Z}}^{\circ} = \left\{
\begin{array}{ll} \mathbb{Z} & \mbox{ if $a\le -4$ and $b\le -4$} \\
(a,b)_{\mathbb{Z}} & \mbox{ if $a\ge -3$ and $b\ge -3$}\\
\end{array} \right.
$$
\end{Theorem}

\begin{proof}Let $\mathcal{C}=(c_{i,j})$ be a non-commutative frieze over $(a,b)_{\mathbb{Z}}$
with 1's on the boundary. 
By Corollary \ref{thm:quidditynorm2} we know that the quiddity sequence contains two
entries with norm less than 2. 

Recall that the norm on $(a,b)_{\mathbb{Z}}$ is defined by
$$\lVert \alpha+\beta\mathrm{i}+\gamma\mathrm{j}+\delta\mathrm{k}\rVert = 
\sqrt{\alpha^2-a\beta^2-b\gamma^2+ab\delta^2}
$$
and the values are roots of non-negative integers since $a,b\in \mathbb{Z}_{< 0}$ by assumption.
Hence
the norm is $<2$ precisely if the integer $\alpha^2-a\beta^2-b\gamma^2+ab\delta^2$ is less than 4. 
\smallskip

We first consider the case $a\le -4$ and $b\le -4$.  Then
the only elements with norm $<2$ are $\pm 1$ and $0$.
A quiddity entry of $\mathcal{C}$ cannot be 0 since by Definition \ref{def:ncpolygon} all entries
in a non-commutative frieze are invertible. Hence $\mathcal{C}$ contains two quiddity entries 
equal to $\pm 1$. 

We know from \cite[Corollary 6.8]{CHJ24} that the quiddity sequence $(c_{0,2},c_{1,3},\ldots,
c_{m-1,m})$ of $\mathcal{C}$ satisfies the formula
\begin{equation} \label{eq:prodmu}
\prod_{r=1}^m \mu(c_{r-1,r+1},c_{r,r+1},c_{r,r-1},c_{r-1,r})= \begin{pmatrix} -1 & 0 \\ 0 & -1 
\end{pmatrix}
\end{equation}
where the non-commutative $\mu$-matrices are defined in \cite[Definition 6.5]{CHJ24} as
$$\mu(c,d,e,f)= \begin{pmatrix} 0 & -e^{-1}d \\ 1 & f^{-1}c \end{pmatrix}.
$$
A general reduction formula for non-commutative quiddity sequences has been shown 
in \cite[Proposition 6.12]{CHJ24}. We apply this formula in our case where all boundary entries
are equal to 1. 

When a quiddity entry $c_{i,i+2}$ is 1 then the formula from \cite[Proposition 6.12]{CHJ24}
takes the form
\begin{equation} \label{eq:prod1111}
\mu(c_{i-1,i+1},1,1,1)\mu(1,1,1,1)\mu(c_{i+1,i+3},1,1,1) =
\mu(-1+c_{i-1,i+1},1,1,1)\mu(-1+c_{i+1,i+3},1,1,1)
\end{equation} 
(see \cite[Remark 6.13]{CHJ24}). This means that we can reduce the product in (\ref{eq:prodmu})
to a product with fewer $\mu$-matrices, that is, we get a shorter non-commutative quiddity cycle
(in the sense of \cite[Definition 6.9]{CHJ24}, i.e. the product of the corresponding $\mu$-matrices
is the negative of the identity matrix). Crucially, the $\mu$-matrices on the right hand side of (\ref{eq:prod1111})
still have the property that the last three entries in the argument (the boundary values) are equal to 1. 
More precisely, the shorter non-commutative 
quiddity cycle is again the quiddity cycle of a 
non-commutative frieze with 1's on the boundary.
In fact, $\mathcal{C}$ is a non-commutative frieze on
an $m$-gon. The restriction of $\mathcal{C}$ to the
$(m-1)$-gon without the vertex $i+1$ is a
non-commutative frieze (by Theorem 
\ref{thm:friezepolygon}), its new boundary entries
are $c_{i,i+2}=1$ and by Proposition 
\ref{rem:equalquddity} also $c_{i+2,i}=1$, and its new
quiddity entries are $c_{i-1,i+2}= -1+c_{i-1,i+1}$
and $c_{i,i+3}=-1+c_{i+1,i+3}$ (these are 
non-commutative exchange relations).
\smallskip

If the quiddity sequence of $\mathcal{C}$ does not contain a 1, then it must contain two entries
equal to $-1$. Note that these $-1$-entries cannot be neighbouring as otherwise the non-commutative
exchange relations give a 0 on a length 3 diagonal which is not allowed since all entries in a non-commutative 
frieze are invertible. 

In the case of a $-1$ quiddity entry the formula from \cite[Proposition 6.12]{CHJ24} reads
\begin{equation} \label{eq:prod11111}
\mu(c_{i-1,i+1},1,1,1)\mu(-1,1,1,1)\mu(c_{i+1,i+3},1,1,1) =
\mu(-1-c_{i-1,i+1},-1,1,1)\mu(-1-c_{i+1,i+3},1,-1,-1).
\end{equation} 
Note that now the $\mu$-matrices on the right hand side no longer have the crucial property that
the last three entries of the argument are equal to 1. However, a direct computation shows that 
$$\mu(-1-c_{i-1,i+1},-1,1,1)\mu(-1-c_{i+1,i+3},1,-1,-1) = 
- \mu(1+c_{i-1,i+1},1,1,1)\mu(1+c_{i+1,i+3},1,1,1).
$$
Thus the above equation (\ref{eq:prod11111}) takes the form
\begin{equation} \label{eq:twistfrieze} \mu(c_{i-1,i+1},1,1,1)\mu(-1,1,1,1)\mu(c_{i+1,i+3},1,1,1)
= - \mu(1+c_{i-1,i+1},1,1,1)\mu(1+c_{i+1,i+3},1,1,1)
\end{equation}
where now the $\mu$-matrices on the right hand side again have the property 
that the last three entries of the argument are equal to 1. But a minus sign appears on the right 
hand side of the formula, hence the shorter sequence after removing the $-1$ quiddity entry
is not a non-commutative quiddity cycle anymore. 

Recall that we are dealing with the case that there are two $-1$ entries in the quiddity sequence
of $\mathcal{C}$ and that these are non-neighbouring. Hence we can apply the reduction
twice, the two minus signs cancel and indeed we get a shorter non-commutative quiddity cycle 
where in all $\mu$-matrices appearing in the product the last three entries in the argument are 1. 
More precisely, this reduction 
produces a non-commutative quiddity cycle which is
the quiddity cycle of a non-commutative frieze on an
$(m-2)$-gon (where $\mathcal{C}$ was a non-commutative
frieze on an $m$-gon). In fact, suppose that
$c_{i,i+2}=-1$ and $c_{k,k+2}=-1$ are two 
non-neighbouring quiddity entries. We remove the
vertices $i+1$ and $k+1$ from the $m$-gon and twist
the remaining entries of $\mathcal{C}$ as follows.
For vertices $r,s$ in the $(m-2)$-gon we set 
$\epsilon_r=-1$ for $r=i+2,i+3,\ldots,k$ and
$\epsilon_r=1$ for $r=k+2,k+3,\ldots,i$ (indices 
modulo $m$) and then
$c'_{r,s}:= \epsilon_r\epsilon_s c_{r,s}$. 
By observing that in each triangle relation and each
non-commutative exchange relation the signs cancel,
one can directly verify that 
$\mathcal{C'}=(c'_{r,s})$ is a non-commutative
frieze on the $(m-2)$-gon; we leave the details to
the reader. For the new boundary entries we have 
$$c'_{i,i+2}=\epsilon_i\epsilon_{i+2}c_{i,i+2}
= 1\cdot (-1)\cdot (-1)=1
$$
and similarly $c'_{k,k+2}=1$; all other boundary 
entries remain 1 by construction of $\mathcal{C'}$. 
For the new quiddity entries we get from the 
non-commutative exchange relations in $\mathcal{C}$ 
that
$$c'_{i-1,i+2} = \epsilon_{i-1}\epsilon_{i+2}
c_{i-1,i+2} = 1\cdot (-1)(-1-c_{i-1,i+1})
= 1+c_{i-1,i+1}
$$
and similarly $c'_{i,i+3}=1+c_{i+1,i+3}$,
$c'_{k-1,k+2}=1+c_{k-1,k+1}$ and 
$c'_{k,k+3}=1+c_{k+1,k+3}$ and these
are precisely the values appearing in the reduction
(\ref{eq:twistfrieze}). The only exceptional case is when 
$k=i+2$, then there is only one new quiddity entry 
$c'_{i,i+4} = c_{i,i+4}=2+c_{i+1,i+3}$, which is also what 
two applications of (\ref{eq:twistfrieze}) give.

\smallskip

We have shown that in both cases (a 1 quiddity entry or two $-1$ quiddity entries) we can reduce 
the non-commutative quiddity cycle coming from $\mathcal{C}$ to a shorter non-commutative
quiddity cycle which is again
the quiddity cycle of a non-commutative frieze
with 1's on the boundary.

Inductively, we can in this way reduce to a non-commutative quiddity cycle of length 2 or 3. 
In \cite[Example 6.11]{CHJ24} we have determined all non-commutative quiddity cycles 
of lengths 2 and 3. The only ones with the property that the last three arguments of the $\mu$-matrices are
1 are the well-known integral quiddity cycles $(0,0)$ and $(1,1,1)$. 
Since the above reductions for $\mu$-matrices only added/subtracted 1 to the quiddity entries $c_{i-1,i+1}$
and $c_{i+1,i+3}$ we can deduce that the original quiddity sequence of $\mathcal{C}$ consists of
integral entries. Then the non-commutative exchange relations (see Definition \ref{def:ncpolygon}) imply
that all entries of $\mathcal{C}$ are rational numbers. On the other hand, the entries of $\mathcal{C}$
are from $(a,b)_{\mathbb{Z}}$. Clearly, $\mathbb{Q}\cap (a,b)_{\mathbb{Z}}= \mathbb{Z}$,
hence $\mathcal{C}$ has only integral entries. A non-commutative frieze over a commutative set is
a classic frieze, thus $\mathcal{C}$ is a Conway-Coxeter frieze or a twisted Conway-Coxeter frieze
(for the classification of non-zero friezes over $\mathbb{Z}$ see \cite{Fontaine}, or 
\cite[Section 4]{CHP24} for an alternative proof). 
\smallskip

We now consider the case $a\ge -3$ and $b\ge -3$, 
i.e. $a,b\in \{-3,-2,-1\}$. By Example 
\ref{ex:ncfriezes-1-2-3} we have non-commutative 
friezes of heights 1 or 3 in which $\mathrm{i}$
and $\mathrm{j}$ appear as entries. In particular,
$\mathrm{i}\in (a,b)^{\circ}_{\mathbb{Z}}$ and
$\mathrm{j}\in (a,b)^{\circ}_{\mathbb{Z}}$.
By Definition \ref{def:friezesubring} the frieze 
subring is the ring generated by all frieze 
entries. This implies that also
$\mathrm{k}=\mathrm{i}\cdot \mathrm{j}
\in (a,b)^{\circ}_{\mathbb{Z}}$. 
It follows that $(a,b)^{\circ}_{\mathbb{Z}}
= (a,b)_{\mathbb{Z}}$ for $a,b\in \{-3,-2,-1\}$. 
\end{proof} 

\medskip

\begin{Remark}
\begin{enumerate}
\item[{(a)}]
This theorem says that for $a\le -4$ and $b\le -4$ the only non-commutative friezes over
the Lipschitz quaternion ring $(a,b)_{\mathbb{Z}}$ are the Conway-Coxeter friezes and the 
twisted Conway-Coxeter friezes. 

Note the analogy to the friezes over rings of integers of imaginary quadratic number fields
(see \cite[Theorem 5.2]{CHP24}): there we had that for a negative square-free integer $d$
we get
$$\mathcal{O}_d^{\circ} = \left\{
\begin{array}{ll} \mathcal{O}_d & \text{ if $d\in \{-1,-2,-3,-7,-11\}$} \\
\mathbb{Z} & \text{ otherwise}
\end{array} \right.
$$
In both situations, we have some small cases where a classification of all friezes seems to be complicated
(although we know that there are only finitely many in each height), 
but in many cases 
the 
friezes over $\mathcal{O}_d$ for $d<0$ and the non-commutative friezes 
over $(a,b)_{\mathbb{Z}}$ for $a<0$ and $b<0$
are known. 
\item[{(b)}] The above theorem does not cover the 
case where exactly one of the (negative) integers
$a$ and $b$ is in $\{-3,-2,-1\}$. 
Computer experiments suggested that in this case the frieze 
subring is $\mathbb{Z}[\mathrm{i}]$ (if $a\in \{-3,-2,-1\}$)
or $\mathbb{Z}[\mathrm{j}]$ (if $b\in \{-3,-2,-1\}$). It seems
to be subtle to prove this as there are plenty of 
possible quiddity cycles (in the sense that the product
of the associated $\mu$-matrices is the negative identity
matrix) involving more than one of 
$\mathrm{i},\mathrm{j},\mathrm{k}$ but they all seem to 
provide zero entries in the corresponding frieze pattern,
which is not allowed. Based on our computer 
experiments we conjectured that 
$(a,b)_{\mathbb{Z}}^{\circ}$ is equal to 
$\mathbb{Z}[\mathrm{i}]$ (if $a\in \{-3,-2,-1\}$)
or $\mathbb{Z}[\mathrm{j}]$ (if $b\in \{-3,-2,-1\}$).
The Appendix contains a proof of this
conjecture found by ChatGPT 6.0 Pro.
\end{enumerate}
\end{Remark}

\section{Appendix}
The aim of this section is to settle the cases left
open in Theorem \ref{thm:abcirc}, that is, to  
determine the non-commutative frieze subrings
of the Lipschitz quaternion rings $(a,b)_{\mathbb{Z}}$
for the cases when one of $a,b$ is in $\{-1,-2,-3\}$
and the other is $\le -4$. 

The proof has been found by ChatGPT 6.0 Pro.
The presentation given below is a rewritten
and carefully checked version by the authors.  
\medskip

The main result of this Appendix is the following. 

\begin{Theorem}\label{thm:main-intro}
Let $a,b<0$ be integers.
\begin{enumerate}
\item If $a\in\{-1,-2,-3\}$ and $b\leq -4$, then
\[
 (a,b)_{\Z}^{\circ}=\Z[\mathrm{i}].
\]
\item If $b\in\{-1,-2,-3\}$ and $a\leq -4$, then
\[
 (a,b)_{\Z}^{\circ}=\Z[\mathrm{j}].
\]
\end{enumerate}
\end{Theorem}
\smallskip

Let us first outline the proof strategy. 
By symmetry we can assume that $a\in\{-1,-2,-3\}$ and $b\leq-4$. We write the Lipschitz quaternion ring as 
$(a,b)_{\Z}=\Z[\mathrm{i}]\oplus\Z[\mathrm{i}]\mathrm{j}$.
The proof has three main steps.  

We first prove a converse to the usual matrix identity for a quiddity cycle:
a sequence $q_0,\ldots,q_{m-1}$ satisfying
\[
 \mu(q_0)\cdots\mu(q_{m-1})= - \begin{pmatrix} 1 & 0 \\
 0 & 1 \end{pmatrix},
 \qquad
 \mu(q)=\begin{pmatrix}0&-1\\1&q\end{pmatrix},
\]
reconstructs a frieze whenever all of its proper cyclic continuants are
nonzero. See Lemma \ref{lem:reconstruction}. 
The nonvanishing condition is essential because frieze entries must
be invertible, see Definition \ref{def:nclocal}.
\smallskip

Secondly, we
show that for each $t\in\mathbb{Z}$, 
the additive map
\[
\delta_t:(a,b)_{\mathbb{Z}}\to (a,b)_{\mathbb{Z}},\,\,\delta_t(x+y\mathrm{j})=x+t y\mathrm{j}
\]
induces a group homomorphism of the group 
$\GE_2((a,b)_{\Z})$ generated by the matrices 
$E(q)=\begin{pmatrix} q & 1 \\ -1 & 0 \end{pmatrix}$
and the invertible diagonal matrices. See Lemma 
\ref{lem:dilation}. 
This is not obvious since the map $\delta_t$ is 
not multiplicative. The proof of Lemma \ref{lem:dilation}
uses a presentation of the group 
$\GE_2((a,b)_{\mathbb{Z}})$ established in
\cite[Proposition~3.1]{BJJKT}, which consists of
Cohn's universal relations \cite{Cohn66} and one crucial additional
relation.
\smallskip

Finally, applying $\delta_t$ to a quiddity cycle gives a one-parameter family
of quiddity cycles.  If one quiddity entry had a nonzero
$\Z[\mathrm{i}]\mathrm{j}$-component, its norm would tend to infinity with $|t|$.  The norm
bound from Theorem \ref{lem:finite} then implies that every sufficiently
large parameter value must make some cyclic continuant vanish.  A polynomial
pigeonhole argument makes one continuant vanish identically, contradicting
the nonvanishing of the original frieze at $t=1$.

\bigskip

\noindent
{\em Step 1: Non-commutative friezes and continuants}

Let $D$ be a division ring, $\mathcal{P}$ an $m$-gon
for an integer $m\ge 3$ and let
$c:\mathrm{diag}(\mathcal{P})\to D^{\ast}$ be a 
non-commutative frieze (as in Definition 
\ref{def:nclocal}). 
We denote the vertices of $\mathcal{P}$ by
$0,1,\ldots,m-1$. The {\em quiddity cycle} of such a non-commutative frieze is
the sequence $(q_0,q_1,\ldots, q_{m-1})$ where
$q_r=c_{r-1,r+1}$ (indices to be read modulo $m$).

The non-commutative exchange relation for the quadrangle with vertices $r,s-1,s,s+1$ give the row-propagation 
formula
\begin{equation}\label{eq:propagation}
 c_{r,s+1}=c_{r,s}q_s-c_{r,s-1},
 \qquad c_{r,r}=0,\quad c_{r,r+1}=1.
\end{equation}  

The quiddity cycle of a non-commutative frieze satisfies
the matrix identity
\begin{equation}\label{eq:frieze-cycle}
 \mu(q_0)\mu(q_1)\cdots\mu(q_{m-1})=-\begin{pmatrix}
     1 & 0 \\ 0 & 1 \end{pmatrix}
\end{equation}
by \cite[Corollary~6.8]{CHJ24}.  We call any sequence 
$(q_o,q_1,\ldots,q_{m-1})$ satisfying
\eqref{eq:frieze-cycle} a \emph{formal quiddity cycle}.

\bigskip

For non-commuting variables define the right continuants by
\begin{equation}\label{eq:continuants}
 K_{-1}=0,\qquad K_0=1,\qquad
 K_\ell(x_1,\ldots,x_\ell)
 =K_{\ell-1}(x_1,\ldots,x_{\ell-1})x_\ell
  -K_{\ell-2}(x_1,\ldots,x_{\ell-2}).
\end{equation}

It is straightforward to check inductively that the 
$\mu$-matrices and the continuants are related by 
the formula
\begin{equation} \label{eq:mucontinuants}
\prod_{i=1}^{\ell} \mu(x_i) = \begin{pmatrix}
-K_{\ell-2}(x_2,\ldots,x_{\ell-1}) & 
-K_{\ell-1}(x_2,\ldots,x_{\ell}) \\
K_{\ell-1}(x_1,\ldots,x_{\ell-1}) &
K_{\ell}(x_1,\ldots,x_{\ell})
\end{pmatrix}    \mbox{\hskip0.4cm for all $\ell\ge 1$}.
\end{equation}

\begin{Lemma}[Reconstruction]\label{lem:reconstruction}
Let $D$ be a division ring, let $m\geq3$, and let
$(q_0,q_1,\ldots,q_{m-1})$ with $q_i\in D$ be a formal quiddity cycle.  Extend the sequence
periodically to $\mathbb{Z}$, and, for $a<b\leq a+m$, set
\begin{equation}\label{eq:candidate-entries}
 c_{a,a}=0,\qquad c_{a,b}
 =K_{b-a-1}(q_{a+1},\ldots,q_{b-1}).
\end{equation}
Suppose that
\begin{equation}\label{eq:proper-nonzero}
 c_{a,b}\neq0\qquad\text{whenever}\qquad 2\leq b-a\leq m-2.
\end{equation}
Then the $c_{a,b}$, read cyclically, define a
non-commutative frieze on the $m$-gon with all boundary values 1, with quiddity cycle $(q_0,q_1,\ldots,q_{m-1})$.
\end{Lemma}

\begin{proof}
The assertion is correct for $m=3$. In fact, the only 
formal quiddity cycle of length 3 is $(1,1,1)$ and 
(\ref{eq:candidate-entries}) gives $c_{a,b}=1$ 
for $a<b$, leading to the only non-commutative frieze 
on a triangle with boundary values 1. 

\smallskip

So we assume from now on that $m\geq 4$.
According to Definition \ref{def:nclocal}\,(c) we have to
show that the values $c_{a,b}$ assigned to the 
edges and diagonals of the $m$-gon as in the lemma
satisfy the local triangle relations and the 
local non-commutative exchange relations. 
\smallskip

We begin with the local triangle relations. 

Since the negative identity matrix is central, every cyclic rotation of a formal quiddity cycle is again
a formal quiddity cycle. Direct expansion using the 
definition of continuants in 
\eqref{eq:continuants} and (\ref{eq:mucontinuants})
gives, for $a+2\leq b\leq a+m$,
\begin{equation}\label{eq:continuant-matrix}
 P_{a,b}:= \mu(q_{a+1})\cdots \mu(q_{b-1})
 =\begin{pmatrix}
   -c_{a+1,b-1}&-c_{a+1,b}\\
    c_{a,b-1}&c_{a,b}
  \end{pmatrix}.
\end{equation}
In particular,
$\mu(q_a) P_{a,a+m}$ is the negative identity matrix 
(being a cyclic rotation of a formal quiddity cycle), 
hence
\[
 P_{a,a+m}=-\mu(q_a)^{-1}
 =\begin{pmatrix}-q_a&-1\\1&0\end{pmatrix}.
\]
Comparison with \eqref{eq:continuant-matrix} yields
\begin{equation}\label{eq:closure}
 c_{a,a+m-1}=c_{a+1,a+m}=1,
 \qquad c_{a,a+m}=0,
 \qquad c_{a+1,a+m-1}=q_a.
\end{equation}
Thus both orientations of every boundary edge have 
value 1 and the
assignment closes cyclically.

For $a+2\leq b\leq a+m-1$, set
\[
 Q_{a,b}:=\mu(q_b)\mu(q_{b+1})\cdots \mu(q_{a+m}).
\]
Then $P_{a,b}Q_{a,b}$ is the negative of the identity matrix, so $P_{a,b}^{-1}=-Q_{a,b}$.  Applying the same
continuant expansion to $Q_{a,b}=P_{b-1,a+m+1}$ as for
(\ref{eq:continuant-matrix}) gives
\[
 P_{a,b}^{-1}
 =\begin{pmatrix}
    c_{b,a+m}&c_{b,a+m+1}\\
   -c_{b-1,a+m}&-c_{b-1,a+m+1}
  \end{pmatrix}.
\]
Reading the subscripts cyclically, we obtain
\begin{equation}\label{eq:inverse-matrix}
 P_{a,b}^{-1}
 =\begin{pmatrix}
    c_{b,a}&c_{b,a+1}\\
   -c_{b-1,a}&-c_{b-1,a+1}
  \end{pmatrix}.
\end{equation}

The lower-left entry of the identity matrix 
$P_{a,b}P_{a,b}^{-1}$ gives
\begin{equation}\label{eq:weak-triangle}
 c_{a,b-1}c_{b,a}=c_{a,b}c_{b-1,a}.
\end{equation}
Because $c_{b,b-1}=c_{b-1,b}=1$, this is precisely the triangle relation for
the vertices $a,b-1,b$ (see Definition \ref{def:ncpolygon}). 
As $a$ and $b$ vary, all weak local triangle
relations (as in \cite[Definition 3.4]{CHJ24}) hold; in particular, all local triangle relations hold.  
\smallskip

We next verify the local non-commutative exchange
relations.

We fix numbers $r<r+1<s<s+1<r+m$ and abbreviate
\[
\begin{array}{llll}
 A=c_{r+1,s},&B=c_{r+1,s+1},&C=c_{r,s},&D=c_{r,s+1},\\
 E=c_{s+1,r},&F=c_{s+1,r+1},&G=c_{s,r},&H=c_{s,r+1}.
\end{array}
\]
Equations \eqref{eq:continuant-matrix} and \eqref{eq:inverse-matrix} read
\[
 P_{r,s+1}=\begin{pmatrix}-A&-B\\ C&D\end{pmatrix},
 \qquad
 P_{r,s+1}^{-1}=\begin{pmatrix}E&F\\-G&-H\end{pmatrix}.
\]
The off-diagonal entries in the identity matrix
$P_{r,s+1}P_{r,s+1}^{-1}=P_{r,s+1}^{-1}P_{r,s+1}$ imply
\begin{align*}
 CE&=DG,& EB&=FD,& AF&=BH,& GA&=HC,\label{eq:off-diagonal-basic}\\
 EG^{-1}&=C^{-1}D,& F^{-1}E&=DB^{-1},&
 HF^{-1}&=B^{-1}A,& G^{-1}H&=AC^{-1}
\end{align*}
where the equations in the second row can be obtained from the 
first row by suitable multiplications.
From this, the four diagonal entries give
$$
 C=F^{-1}+DB^{-1}A,\,\,
 G=B^{-1}+HF^{-1}E,\,\,
 B=G^{-1}+AC^{-1}D,\,\,
 F=C^{-1}+EG^{-1}H.
$$
These are exactly the four local non-commutative
exchange relations for the quadrangle
with vertices $r,r+1,s,s+1$ after inserting its four boundary values, all equal to one.
Every element inverted here is nonzero by assumption
\eqref{eq:proper-nonzero} or is a
boundary value. 
\smallskip

As all local triangle relations and all local 
non-commutative exchange relations hold, the 
reconstructed assignment is
a frieze.  Its quiddity is $(q_0,q_1,\ldots, q_{m-1})$ by
\eqref{eq:candidate-entries}.
\end{proof}
\medskip

\noindent
{\em Step 2: The group homomorphism $\Delta_t$ on
the group $\mathrm{GE}_2((a,b)_{\mathbb{Z}})$.}

We continue to assume that
$ a\in\{-1,-2,-3\}$ and $b\leq-4$
and consider again
the decomposition 
$(a,b)_{\mathbb{Z}} = \mathbb{Z}[\mathrm{i}]\oplus
\mathbb{Z}[\mathrm{i}]\mathrm{j}$.
For all $x\in \mathbb{Z}[\mathrm{i}]$ we have
\begin{equation}\label{eq:twisting}
 \mathrm{j}x=\overline{x}\mathrm{j}.
\end{equation}
In fact, write $x=\alpha + \beta \mathrm{i}$. Then
$$\mathrm{j}x = \alpha \mathrm{j}- \beta \mathrm{k}
= (\alpha-\beta \mathrm{i})\mathrm{j} = \overline{x}
\mathrm{j}.
$$

For $t\in\Z$ define the {\em additive transverse dilation}
by setting
$$ \delta_t:(a,b)_{\mathbb{Z}}\to (a,b)_{\mathbb{Z}},
 \,\,\,\delta_t(x+y\mathrm{j})=x+t y\mathrm{j}
 \text{\hskip0.4cm for $x,y\in \mathbb{Z}[\mathrm{i}]$}.
$$

We consider the matrices
$E(q):=\begin{pmatrix}q&1\\-1&0\end{pmatrix}$;
they are related to the $\mu$-matrices by the equation
$\mu(q)=E(q)^{-1}$. Moreover, we write 
$[u,v]$ for the diagonal matrix with diagonal 
entries $u$ and $v$. 

The group $\GE_2((a,b)_{\mathbb{Z}})$ is defined as the group generated by the matrices 
$E(q)$ together with the invertible diagonal matrices.

\begin{Lemma}[Transverse dilation]\label{lem:dilation}
For every $t\in\Z$ there is a group homomorphism
\[
 \Delta_t:\GE_2((a,b)_{\mathbb{Z}})\longrightarrow
 \GE_2((a,b)_{\mathbb{Z}})
\]
such that
\[
 \Delta_t(E(q))=E(\delta_t(q))\quad(q\in (a,b)_{\mathbb{Z}})\,\,\,
 \text{~~~and~~~}\,\,\,
 \Delta_t([u,v])=[u,v]\quad(u,v\in 
 (a,b)_{\mathbb{Z}}^{\ast}).
\]
Consequently, every formal quiddity cycle $(q_0,\ldots,q_{m-1})$ over
$(a,b)_{\mathbb{Z}}$ gives another formal quiddity cycle
$(\delta_t(q_0),\ldots,\delta_t(q_{m-1}))$.
\end{Lemma}

\begin{proof}
Recall the norm form on $(a,b)_{\mathbb{Z}}$
from Definition \ref{def:norm}. For an arbitrary
element $x+y\mathrm{j}\in \mathbb{Z}[\mathrm{i}]
\oplus \mathbb{Z}[\mathrm{i}]\mathrm{j}$
(with $x,y\in \mathbb{Z}[\mathrm{i}])$, the norm 
form can quickly be computed to satisfy
\begin{equation}\label{eq:norm-splitting}
 N(x+yj)=x\overline{x}-b\,y\overline{y}.
\end{equation}
If $q\in (a,b)_{\mathbb{Z}}^{\ast}$ is a unit, 
then $N(q)=1$ (since the positive integers $N(q)$ 
and
$N(q^{-1})$ have product 1).  Since $b\leq -4$, formula
\eqref{eq:norm-splitting} forces the 
$\mathbb{Z}[\mathrm{j}]$-component of $q$ to vanish.  Hence
\begin{equation*}\label{eq:units-in-C}
 (a,b)_{\mathbb{Z}}^{\ast} =
 \mathbb{Z}[\mathrm{i}]^{\ast}
 =\begin{cases}
   \{\pm1,\pm i\},& \text{if~}a=-1,\\
   \{\pm1\},& \text{if~}a=-2,-3.
  \end{cases}
\end{equation*}
The same argument shows that
\begin{equation}\label{eq:short-in-C}
 1<\lVert q\rVert<2\quad\Longrightarrow\quad 
 q\in \mathbb{Z}[\mathrm{i}],
\end{equation}
because then $N(q)\in\{2,3\}$.  The map $\delta_t$ therefore fixes every
unit and every element of norm strictly between one and two, as well as its
quaternionic conjugate.

For $u,v\in (a,b)_{\mathbb{Z}}^{\ast}=
\mathbb{Z}[\mathrm{i}]^{\ast}$ and 
$q=x+yj\in \mathbb{Z}[\mathrm{i}]\oplus \mathbb{Z}[\mathrm{i}]\mathrm{j}=(a,b)_{\mathbb{Z}}$, equation
\eqref{eq:twisting} gives
\begin{equation}\label{eq:unit-equivariance}
 \delta_t(uqv)
 =\delta_t(uxv+uy\bar vj)
 =uxv+tuy\bar vj
 =u\delta_t(q)v.
\end{equation}

We now invoke \cite[Proposition~3.1]{BJJKT}.  It states that, for an order in
a totally definite rational quaternion algebra, a complete presentation of the group
$\GE_2$ consists of Cohn's universal relations 
(i.e. relations \cite[(R1)-(R4)]{BJJKT} together with 
the relations in $D_2((a,b)_{\mathbb{Z}})$, the multiplicative group of invertible $2\times 2$-matrices
over $(a,b)_{\mathbb{Z}}$; see \cite{Cohn66}) 
plus the relation
\begin{equation}\label{eq:short-relators}
 \bigl(E(\bar z)E(z)\bigr)^n=E(0)^2
 \qquad\text{for }1<\lVert z\rVert=\sqrt n<2.
\end{equation}
The function $\Delta_t$ in the statement preserves the additive universal relation (R1),
\[
\hskip1cm E(x)E(0)E(y)=E(0)^2E(x+y),
\]
because $\delta_t$ is additive.  It preserves all universal relations
of types (R2) (for a unit $x$),
$$E(x)E(x^{-1})E(x) = E(0)^2[x,x^{-1}]
$$
and (R4), 
$$E(0)^2 = [-1,-1]
$$
as well as the relations in
$D_2((a,b)_{\mathbb{Z}}$, 
because zero and each unit is fixed by $\delta_t$.  It preserves the remaining
universal relation (R3),
\[
 E(x)[u,v]=[v,u]E(v^{-1}xu)
\]
by \eqref{eq:unit-equivariance}.  
Finally, $\delta_t$ maps each extra relation
\eqref{eq:short-relators} to itself by
\eqref{eq:short-in-C} (because $\delta_t(z)=z$ for all
$z\in \mathbb{Z}[\mathrm{i}]$).  The complete presentation therefore yields that the
map $\Delta_t$ is a group homomorphism.

For the final statement, if 
$\mu(q_0)\cdots\mu(q_{m-1})$ is the negative identity
matrix, apply $\Delta_t$, use
$\mu(q)=E(q)^{-1}$, and note that 
$E(0)^2=-\begin{pmatrix} 1 & 0 \\ 0 & 1 \end{pmatrix}$ is fixed.  This gives that also 
\[
 \mu(\delta_t(q_0))\cdots\mu(\delta_t(q_{m-1}))
\]
is the negative identity matrix, as required.
\end{proof}

\begin{Remark}
The complete presentation is crucial in Lemma~\ref{lem:dilation}.
The map $\delta_t$ is generally not a ring homomorphism, and entrywise
application of $\delta_t$ to a matrix is not a matrix-group homomorphism.
\end{Remark}

\smallskip

\noindent
{\em Step 3: The frieze subring}

If the quiddity entries $q_r=q_r(t)$ depend polynomially on a central variable $t$, then every
candidate entry in \eqref{eq:candidate-entries} is likewise a polynomial in
$t$.  Lemma~\ref{lem:reconstruction} says that the only obstruction to
obtaining a genuine non-commutative frieze from a formal 
quiddity cycle is the vanishing of one of the
finitely many proper cyclic continuants.

\bigskip

\begin{Proposition}\label{thm:main}
Assume $a\in \{-1,-2,-3\}$ and $b\le -4$.  Every 
non-commutative frieze over $(a,b)_{\mathbb{Q}}$ with
all boundary entries 1
whose
entries lie in the Lipschitz subring $(a,b)_{\mathbb{Z}}\setminus\{0\}$ has all its entries in $\mathbb{Z}[\mathrm{i}]$.
Consequently, the frieze subring is
\[
 (a,b)_{\mathbb{Z}}^{\circ}=\mathbb{Z}[\mathrm{i}].
\]
\end{Proposition}

\begin{proof}
Let $c$ be such a non-commutative frieze on an $m$-gon, 
and let $(q_0,\ldots,q_{m-1})$ be its quiddity cycle.  If $m=3$
there are only boundary entries, so assume $m\geq 4$. 
Each quiddity entry is in $(a,b)_{\mathbb{Z}}=\mathbb{Z}[\mathrm{i}]\oplus \mathbb{Z}[\mathrm{i}]\mathrm{j}$, so
we write
\[
 q_r=x_r+y_r \mathrm{j}\qquad(x_r,y_r\in \mathbb{Z}[\mathrm{i}]).
\]
For every $t\in\Z$, put
\[
 q_r(t)=\delta_t(q_r)=x_r+t y_r \mathrm{j}.
\]
By Lemma~\ref{lem:dilation}, the sequence $(q_r(t))$ is a formal quiddity
cycle for every $t$.

Use \eqref{eq:candidate-entries} to form all proper cyclic continuants
$c_{u,v}(t)$ from this formal quiddity cycle.  Each is an $(a,b)_{\mathbb{Q}}$-valued polynomial in the
central variable $t$, with coefficients in $(a,b)_{\mathbb{Z}}$.  At $t=1$, propagation
\eqref{eq:propagation} shows that these candidate values are exactly the
entries of the original frieze.  In particular,
\begin{equation}\label{eq:nonzero-at-one}
 c_{u,v}(1)\neq0
 \qquad(2\leq v-u\leq m-2).
\end{equation}

Suppose, for a contradiction, that $y_r\neq0$ for some $r$.  Formula
\eqref{eq:norm-splitting} gives
\begin{equation}\label{eq:growing-norm}
 \lVert q_r(t)\rVert^2=x_r\bar x_r-bt^2y_r\bar y_r,
\end{equation}
so $\lVert q_r(t)\rVert\to\infty$ as $|t|\to\infty$.  In particular,
$\lVert q_r(t)\rVert>m-2$ for every sufficiently large $|t|$.

For each such integer $t$, at least one proper cyclic continuant must vanish.
Indeed, if they were all nonzero, Lemma~\ref{lem:reconstruction} would turn
the formal quiddity cycle into a non-commutative frieze over
$(a,b)_{\mathbb{Z}}\setminus\{0\}$, contrary to the 
norm-bound in 
Theorem \ref{lem:finite} and
\eqref{eq:growing-norm}.  There are only finitely many proper cyclic
continuants.  Hence one fixed polynomial $c_{u,v}(t)$ vanishes for infinitely
many integers $t$.  Expressing its coefficients in the $\Q$-basis
$1,\mathrm{i},\mathrm{j},\mathrm{k}$ of 
$(a,b)_{\mathbb{Q}}$ shows that each of its four scalar coordinate polynomials
vanishes identically.  Thus $c_{u,v}(t)$ is the zero polynomial, which
contradicts \eqref{eq:nonzero-at-one}.

We conclude that $y_r=0$ for every $r$, so every quiddity entry lies in $\mathbb{Z}[\mathrm{i}]$.
The propagation formula \eqref{eq:propagation}, starting from $0$ and $1$,
then shows that every frieze entry lies in 
$\mathbb{Z}[\mathrm{i}]$.  This proves
$(a,b)_{\mathbb{Z}}^{\circ}\subseteq \mathbb{Z}[\mathrm{i}]$.

For the reverse inclusion, the boundary entries show that
$\mathbb{Z}\subseteq (a,b)_{\mathbb{Z}}^{\circ}$.  The 
non-commutative friezes in Example \ref{ex:ncfriezes-1-2-3}
show that $\mathrm{i}$ is a frieze entry in each of
the remaining cases: for $a=-1$ use the height-one frieze with alternating
entries $2\mathrm{i},-\mathrm{i}$; for $a=-2$ use the height-one frieze with alternating
entries $\mathrm{i},-\mathrm{i}$; and for $a=-3$ use the height-three frieze whose quiddity
entries alternate between $\mathrm{i}$ and $-\mathrm{i}$ and whose length-three entries are
$2$.  Hence $\mathbb{Z}[\mathrm{i}]\subseteq
(a,b)_{\mathbb{Z}}^{\circ}$, completing the proof.
\end{proof}

Interchanging $\mathrm{i}$ and $\mathrm{j}$ in Proposition~\ref{thm:main} proves the second half of
Theorem~\ref{thm:main-intro}.

Combining Theorem~\ref{thm:main-intro} with
Theorem \ref{thm:abcirc} gives the full answer for all non-commutative frieze subrings of Lipschitz quaternion 
rings $(a,b)_{\mathbb{Z}}$ with negative
integral parameters.

\begin{Theorem}\label{cor:classification}
Let $a,b<0$ be integers.  Then
$$(a,b)_{\mathbb{Z}}^{\circ} = \left\{
\begin{array}{ll} \mathbb{Z} & \mbox{ if $a\le -4$ and $b\le -4$}, \\
(a,b)_{\mathbb{Z}} & \mbox{ if $a\ge -3$ and $b\ge -3$}, \\
\mathbb{Z}[\mathrm{i}] & \mbox{ if $a\in \{-1,-2,-3\}$
and $b\le -4$}, \\
\mathbb{Z}[\mathrm{j}] & \mbox{ if $a\le -4$ and $b\in \{-1,-2,-3\}$}.
\end{array} \right.
$$
\end{Theorem}

\bigskip

\noindent
{\bf Statement on the use of AI.}
The paper has been largely written without the use of 
AI. Only at the very final stages we have involved
ChatGPT 6.0 Pro. The model confirmed our conjecture on the 
frieze subrings in the cases $a\in \{-1,-2,-3\}$
and $b\le -4$, or vice versa, see Theorem \ref{cor:classification}. The proof the model provided
is documented in the Appendix; it has been carefully 
checked and partially rewritten by the authors. In 
connection with another project on commutative friezes ChatGPT suggested a better bound and a simpler proof
of a commutative version of Theorem \ref{lem:finite}.
We used the idea to generalize it to the non-commutative
setting which is now included here. Claude Fable 5.1
Max assisted our proofreading.
In this process, Claude observed a small gap in the 
proof of Theorem \ref{thm:abcirc} and suggested 
an easy fix
which we included.
The authors take full responsibility for all aspects 
of the content and the correctness of the paper. 

\bigskip

\noindent
{\bf Acknowledgement.} The third author was supported 
by Aarhus University Research Foundation, grant 
number AUFF-E-2024-9-43.


\end{document}